%% file: indefinite_geneo_revision.tex
\documentclass[11pt]{article}

\usepackage[a4paper]{geometry}
\usepackage[show]{ed}   
\newcommand{\iggr}[1]{{\color{black}{#1}}}
\newcommand{\nb}[1]{{\color{black}{#1}}}

\usepackage{amsthm}   
\usepackage{xcolor}
\usepackage{url}
\usepackage{tabu}
\usepackage{bm}

\newtheorem{theorem}{Theorem}[section]
\newtheorem{assumption}[theorem]{Assumption}
\newtheorem{corollary}[theorem]{Corollary}
\newtheorem{definition}[theorem]{Definition}
\newtheorem{notation}[theorem]{Notation}
\newtheorem{example}[theorem]{Example}
\newtheorem{lemma}[theorem]{Lemma}

\newtheorem{proposition}[theorem]{Proposition}

\newcommand{\Cstab}{C_{\mathrm{stab}}}
\newcommand{\cinf}{c^-_{\inf}}

\usepackage{}
\usepackage{graphicx}
\usepackage{multirow,bigstrut}
\usepackage{amsmath,amsfonts,amssymb}
\usepackage{mathrsfs}
\usepackage{color}
\usepackage{upgreek}
\usepackage{bm}
\usepackage{verbatim}
\usepackage{mathtools}
\usepackage{calligra}
\usepackage[T1]{fontenc}
\usepackage{subfigure}
\usepackage{latexsym}
\usepackage{color}
\usepackage{soul}
\usepackage{authblk}

\usepackage{verbatim}
\usepackage{exscale}
\usepackage{relsize}

\usepackage{tikz}
\usetikzlibrary{shapes,arrows}

\usepackage{authblk}
\usepackage{hyperref}

\newcommand{\tmop}[1]{\ensuremath{\operatorname{#1}}}

\newcommand{\tV}{\widetilde{V}}
\numberwithin{equation}{section}

\begin{document}


\title{Overlapping Schwarz methods with GenEO coarse spaces for
  indefinite and non-self-adjoint problems}

\author{Niall Bootland\protect\footnote{Department of Mathematics and
    Statistics, University of Strathclyde, Glasgow, G1 1XH, UK; {\tt
      niall.bootland@strath.ac.uk, work@victoritadolean.com}},
Victorita Dolean$^{\ast,}$\protect\footnote{Laboratoire
  J.A.~Dieudonn\'e, CNRS, University C\^ote d'Azur, Nice, 06108,
  France}, Ivan G. Graham\protect\footnote {Department of Mathematical
  Sciences, University of Bath, Bath, BA2 7AY, UK; {\tt i.g.graham@bath.ac.uk}},
  Chupeng Ma\protect\footnote{Institute for Applied Mathematics \&
    Interdisciplinary Center for Scientific Computing, Heidelberg
    University, 69120 Heidelberg, Germany; {\tt
      chupeng.ma@uni-heidelberg.de, r.scheichl@uni-heidelberg.de}}, Robert Scheichl$^{\ddag,\S}$}

\date{} 


\maketitle

\begin{abstract}
\noindent
GenEO (`Generalised Eigenvalue problems on the Overlap') is a method for computing an operator-dependent spectral coarse space to be combined with local solves on subdomains to form a robust parallel domain decomposition preconditioner for elliptic PDEs.  It has previously been proved, in the self-adjoint and positive-definite case, that this method, when used as a preconditioner for conjugate gradients, yields iteration numbers which are completely independent of the heterogeneity of the coefficient field of the partial differential operator. We extend this theory to the case of convection--diffusion--reaction problems, which may be non-self-adjoint and indefinite, and whose discretisations are solved with preconditioned GMRES. The GenEO coarse space is defined here using a generalised eigenvalue problem based on a self-adjoint and positive-definite subproblem. \iggr{We prove estimates on
GMRES iteration counts which are independent of the variation of the coefficient of the diffusion term in the operator and depend only very mildly on  variations of the other coefficients. These are proved  under the assumption that the subdomain diameter is sufficiently small and the eigenvalue tolerance for building the coarse space is sufficiently large.}  While the iteration number estimates do grow as the non-self-adjointness and indefiniteness of the operator increases, practical tests indicate the deterioration is much milder. Thus we obtain an iterative solver which is efficient in parallel and very effective for a wide range of convection--diffusion--reaction problems.
\end{abstract}

\noindent
{\bf MSC Classes:} {\tt 65N22, 65N55, 65F10}
\\[2ex]

\input{introduction_revision.tex}

\input{domain_decomposition.tex}

\input{theoretical_tools.tex}

\input{main_proof.tex}

\input{numerical_results_revision.tex}


\end{document}

%% file: introduction_revision.tex
\section{Introduction}\label{sec-1}

This paper is concerned with the theory and implementation of two-level domain decomposition preconditioners for finite element  systems arising from indefinite and/or non-self-adjoint elliptic boundary value problems with rough or highly variable coefficients.  We analyse and test a method which is scalable with respect to the number of subdomains and also robust to  heterogeneity, backed up by supporting theory. The method is developed for the following
scalar second-order (convection--diffusion--reaction) PDE with homogeneous Dirichlet boundary conditions:
\begin{subequations}
\label{eq:2-1}
\begin{align}
-\mathrm{div}( A \nabla u) + \mathbf{b}\cdot\nabla u + cu &= f & & \text{in } \Omega, \\
u &= 0 & & \text{on } \partial\Omega,
\end{align}
\end{subequations}
posed on a bounded  polygonal or Lipschitz  polyhedral domain $\Omega\subset \mathbb{R}^{d}$ ($d=2,3$). As well as a weak regularity assumption on the
coefficients $A, \mathbf{b}$ and $c$ (Assumption \ref{ass:1}), we will assume that
\eqref{eq:2-1}  has a unique weak solution $u\in H^1(\Omega)$ for all $f \in L^2(\Omega)$.

In overlapping domain  decomposition methods, the global domain $\Omega$ is covered by a set of
overlapping subdomains~$\Omega_i$, $i=1,\ldots N$,  and the  classical one-level additive Schwarz preconditioner
is built from partial solutions on each subdomain. Since this preconditioner is not in general scalable as the number of subdomains grows,  an  additional global  coarse solve  is usually  added to enhance scalability, as well as  robustness
with respect to coefficient heterogeneity.

\bigskip

\noindent {\bf Aim of the paper and related literature.}
The construction of suitable coarse spaces has enjoyed intensive recent  interest. Since classical
coarse spaces (piecewise polynomials on a coarse grid, e.g., \cite{cai,toselli}) are not
robust for highly heterogeneous problems,
a number of groups have developed operator-dependent coarse spaces
which are better adapted to the heterogeneity, an early example being \cite{ivanrob}.
Most recently,  attention has focussed on
`spectral' coarse spaces  which are built from selected modes of local generalised eigenvalue problems.
The first proposals of this technique for domain decomposition preconditioning  were \cite{galvis1,galvis2}.  The use of local spectral information is important in   the related but different context of robust approximation techniques
for multiscale PDE problems, e.g., \cite{babuska,efendiev,ma1,ma2}.

This paper focuses on the GenEO coarse space, first proposed in
\cite{spillane:2014} for    general coercive, self-adjoint systems of PDEs (e.g., \eqref{eq:2-1}
with $\mathbf{b} =0$ and $c \geq 0$).   In \cite{spillane:2014},  it was
rigorously proved that (when combined with a one-level method and using  preconditioned conjugate gradients as the iterative solver), the resulting algorithm enjoys not only
scalability with respect to  the number of subdomains, but also robustness with respect to coefficient variation.     Recent contributions in this very active field  include \cite{giraud, klawonn,peter} and \cite{spillane:2021}, while  more  complete lists of contributions can be found in the literature surveys in, e.g., \iggr{in the introductions to  \cite{spillane:2014} or  \cite{galvis:2018:ODD}.}

While all the methods cited so far were developed for SPD problems,  where the behaviour of the  local
eigenspaces  can be more readily  analysed, there has also been considerable (so far mostly empirical)  progress on applying related techniques to non-coercive problems. Most of this has been for
the Helmholtz problem at high frequency, although \cite{nataf:2021} contains \iggr{rigorous} recent work on saddle-point systems.
For Helmholtz, the first publications include \cite{conen:2014,bootland:2021:OTD}, while recent work in
\cite{bootland:2021:ACS} compares the performance of several spectral coarse spaces
for high-frequency Helmholtz problems in a high performance environment. One of the most encouraging coarse spaces for Helmholtz problems
(called `H-GenEO' in \cite{bootland:2021:ACS,bootland:2021:GCS}) has  promising robustness properties at high frequency, but, since it  involves solving for local generalised eigenvalues of the Helmholtz operator,
it is beyond the reach of the present theory. However, the fact that
H-GenEO behaves similarly to the method analysed in this paper at lower frequencies, see \cite{bootland:2021:GCS}, motivated the current work.

Our other primary motivation was  to find out how well the  GenEO technology works (both in theory and in practice) for general  non-self-adjoint and indefinite PDEs of the form \eqref{eq:2-1}.
The GenEO coarse space that we analyse here is the same as the space
previously proposed and analysed in the SPD case. Thus our numerical results here indicate  the results which  can be expected when any standard GenEO code
is used for the more general problem \eqref{eq:2-1}.
The main contribution of the present paper therefore is the analysis of how GenEO
performs as a preconditioner for GMRES when applied to \eqref{eq:2-1}
and, in particular, its  dependence  on heterogeneity, indefiniteness, and lack of self-adjointness.

Since the FE systems that arise from \eqref{eq:2-1} are  in general non-Hermitian,
we use  GMRES as the iterative solver
and our main tool for convergence  analysis is the  `Elman theory'
(\cite{elman} or \cite[Lemma C.11]{toselli}), which requires   an upper bound for the norm of
the preconditioned matrix and a lower  bound on the distance of its field of values from the origin.
The number of GMRES iterates to achieve a given error tolerance is then a function of these two bounds.
Thus, the dependence  of these bounds on the parameters of the problem (e.g., mesh size, subdomain
size, overlap, heterogeneity, and strength of indefiniteness/non-self-adjointness) is of paramount
interest when analysing the robustness and scalability of the preconditioner.
The use of the Elman theory  in domain decomposition goes back to Cai and Widlund  \cite{cai} (see also \cite[Section 11]{toselli}), who used
classical linear piecewise polynomials on a coarse grid
without attention to coefficient heterogeneity.
Analogous  results for Maxwell's equations were obtained in \cite{pasciak}.

\bigskip

\noindent {\bf The main results of this paper.}  Our main theoretical result, Theorem~\ref{thm:2-1}, provides rigorous upper bounds on the coarse mesh
diameter $H$ and on the `eigenvalue tolerance'~$\Theta$ (for the local GenEO eigenproblems) which ensure
that GMRES enjoys robust and mesh-independent convergence  when applied to the preconditioned problem.
Reducing   $H$  and $\Theta$ still  further would  improve  GMRES convergence  but increase  cost:
smaller     $\Theta$ means that more  local eigenfunctions are incorporated into the coarse space;
smaller  $H$ means that more
local eigenvalue problems have to be solved, although the size of the local subdomains becomes smaller and, in practice,  fewer eigenfunctions are needed on each subdomain. Optimal choices of $\Theta$ and $H$  are typically determined empirically.
Since Theorem \ref{thm:2-1}  is quite technical, by way of introduction we present here
the result when it is specialised to the following simpler, yet instructive,   example.

\begin{example}[An indefinite but self-adjoint problem]   \label{ex:1} \em
Consider problem \eqref{eq:2-1} with $\mathbf{b} = \mathbf{0}$, and $c(x) = - \kappa$ where $\kappa$ is a positive constant. That is
\begin{align}
-{\rm div}( A \nabla u) -  \kappa u = f,
\label{DP}
\end{align}
with homogeneous Dirichlet boundary conditions.  Let  $a_{\min}$ and $a_{\max}$ be, respectively, lower and upper spectral bounds on $A$ (see \eqref{eq:Abound}). Without loss of generality,  we assume that $a_{\min} = 1$ and that the diameter $D_\Omega$ of $\Omega$ satisfies $D_\Omega \le 1$. (Otherwise the problem can be rescaled to achieve these requirements.) The GenEO coarse space is based on the $m_i$ (dominant) eigenfunctions corresponding to the smallest
eigenvalues \iggr{$\lambda^i_1 \leq \lambda^i_2 \leq \ldots \leq \lambda_{m_i}^i$} of the generalised eigenvalue problem   on  $\Omega_i $  defined in \eqref{eq:2-18}.  To obtain a robust rate of convergence for GMRES that depends only on $k_0$ (the maximum  number of times any point is overlapped by the subdomains $\Omega_i$), Theorem \ref{thm:2-1} proves that
the following conditions on the `eigenvalue tolerance' $\Theta := 1/\min_{i=1}^N  \lambda^i_{m_i+1} $ and on the  (coarse) mesh diameter $H$ are sufficient:
\begin{align} \label{eq:bdLam0}
H \ \lesssim \ \kappa^{-1} \quad \text{and} \quad  \Theta^{1/2} \lesssim (\Cstab^* + 1)^{-1} \,  \kappa^{-2}.
\end{align}
Here   $\Cstab^*>0$ is the stability constant for the adjoint problem to \eqref{eq:2-1} and is
defined in \eqref{eq:Cstabstar}. It blows up as  the distance of $\kappa$ from an eigenvalue of the Dirichlet problem for $-{\rm div} (A \nabla u)$ decreases.
The hidden constants in \eqref{eq:bdLam0}
\iggr{depend only  on  $k_0$}.
Thus, under  conditions \eqref{eq:bdLam0} the rate of convergence of GMRES depends only on $k_0$.
For constant $A$, the first  condition in \eqref{eq:bdLam0} is exactly the one required to ensure that the piecewise linear finite element discretization of \eqref{eq:2-1} on a coarse mesh of diameter $H$ is uniquely solvable (see, e.g., \cite{cai}).
\end{example}

The full statement of Theorem \ref{thm:2-1}  for a spatially- and sign-varying $c$ and for a nonzero, spatially varying  convection coefficient  $\mathbf{b}$ is structurally  similar. Sufficient conditions for robustness are obtained by replacing   $\kappa^{-1}$ and $\kappa^{-2}$  in the two bounds  \eqref{eq:bdLam0}, respectively, by
$$(1 +  \|\mathbf{b}\|_{\infty} + \|c\|_{\infty}  )^{-1},  $$
and
$$(1+  \|\mathbf{b}\|_{\infty}  + \|c\|_{\infty}  )^{-1}(1+ \|\mathbf{b}\|_{\infty} + \|\nabla \cdot \mathbf{b}\|_{\infty} + \|c\|_{\infty} )^{-1}.$$
By making a small change to the GenEO construction,  the contribution from $c$ in  these bounds can be removed if $c$ is uniformly non-negative. We remark also that our results do not depend on the spatial variation of $A$  or $c$ and depend on the variation of $\mathbf{b}$ only through its divergence.

The proof of Theorem \ref{thm:2-1} is obtained by estimating the norm and the field of values of the preconditioned problem in the energy inner product induced by the second order operator in \eqref{eq:2-1},  and then applying the Elman estimate. Several technical difficulties have to be overcome, particularly:  (i) since the regularity assumptions on the coefficients $A$, $\mathbf{b}$, and $c$ are minimal, there is no a priori regularity for solutions of \eqref{eq:2-1} beyond $H^1$; (ii) because of the indefiniteness and non-self-adjointness of \eqref{eq:2-1}, the existence and stability of solutions of both local and coarse space subproblems
has to be established,  and this is the first source of conditions on $H$ and $\Theta$; (iii) all estimates should be explicit in the coefficients $A$, $\mathbf{b}$, and $c$, and as sharp as possible.

After the proof of  Theorem \ref{thm:2-1}, we remark  that restricting the result to the special case $\mathbf{b} = \mathbf{0}$ and $c = 0$, we obtain essentially the same result as in \cite{spillane:2014}, which is interesting because the latter makes use of `self-adjoint' technology, working with eigenvalue and condition number estimates and does not estimate the field of values as is required in \cite{elman}.

In Section \ref{sec:numerical} we give substantial numerical experiments illustrating the theory. All experiments are on rectangular domains with uniform fine meshes.  Highlights of these results include: (i) for  problem \eqref{eq:2-1},  with $A = I$,  $\mathbf{b} = 0$, and $c = -\kappa$ constant,  the standard GenEO preconditioner is remarkably effective even up to about $\kappa = 1000$; (ii) in the case $A = I$, $c = 0$, and $\mathbf{b}$ variable, the method is effective up to about $\Vert \mathbf{b} \Vert_\infty = 1000$, with the results only degrading moderately as the divergence of
$\mathbf{b}$ increases from zero; (iii) the results above  hardly change when $A = I$ is changed to a highly heterogeneous field; (iv) provided the ratio of  overlap width to subdomain diameter is held fixed as $h\rightarrow 0$, the dimension of the GenEO coarse space is independent of $h$; (v) the method provides a
very effective preconditioner for systems arising from implicit time-stepping applied to convection--diffusion--reaction problems, and the preconditioner can be assembled once and used for all choices of (variable) time-steps.

We remark that, while Theorem \ref{thm:2-1} provides convergence estimates for  `weighted'  GMRES
applied  in the `energy' inner product, our numerical experiments use  standard GMRES and the    Euclidean inner product. In Corollary \ref{cor:Euclidean}
we show that standard GMRES attains the same accuracy as weighted GMRES after
a few extra iterations which grow only logarithmically in $h^{-1}$ and in the contrast of the matrix $A$.
The empirical observation that weighted and standard GMRES yield similar iteration counts in domain decomposition methods has been made before, e.g., \cite{cai_zou,ivan} and a result related but not identical to Corollary \ref{cor:Euclidean} is given in \cite[Corollary 5.8]{shihua}.

In our experiments the rather stringent looking condition on the eigenvalue tolerance in \eqref{eq:bdLam0} appears to be overly pessimistic and the coarse space sizes necessary to achieve full robustness in practice
are moderate. Similar observations were made in \cite{cai}, for homogeneous but indefinite problems and classical coarse spaces.

\bigskip

\noindent {\bf Plan of the paper.} In Section \ref{prob} we give some background to domain decomposition for
\eqref{eq:2-1} and introduce the GenEO preconditioner. In Section \ref{sec-2} we prove some of the
central technical estimates  concerning solutions of the subproblems, which are then used in Section \ref{Proof} to prove the main theorem. Finally, numerical experiments are detailed in Section \ref{sec:numerical}.

%% file: domain_decomposition.tex
\section{Background Material}
\label{prob}

\subsection{Problem formulation and discretization}

The weak formulation of (\ref{eq:2-1}) is to find $u\in H_{0}^{1}(\Omega)$ such that
\begin{equation}\label{eq:2-2}
b(u,v) = (f,v),\quad \forall v\in H_{0}^{1}(\Omega),
\end{equation}
where   $f\in L^2(\Omega)$ and  the bilinear form $b(\cdot,\cdot) \colon H_{0}^{1}(\Omega)\times H_{0}^{1}(\Omega)\rightarrow \mathbb{R}$ is defined by
\begin{equation}\label{eq:2-3}
b(u,v) = \int_{\Omega}\big(A\nabla u\cdot{\nabla v} + \mathbf{b}\cdot \nabla u \, v+ cu v\big)dx.
\end{equation}
Throughout the paper we make the following assumptions on the problem:

\begin{assumption} \label{ass:1}
The coefficients ($A$, $\mathbf{b}$, $c$) and the right-hand side $f$ are assumed to satisfy
\begin{enumerate}
\item[(i)]
$A:\Omega\rightarrow \mathbb{R}^{d\times d}$ is symmetric and
there exists $0< a_{\rm min} < a_{\rm max}$ such that
\begin{equation}
a_{\rm min} |\xi|^{2}\leq A(x)\xi\cdot\xi \leq a_{\rm max} |\xi|^{2}, \quad \text{for all} \ x\in\Omega,\;\;  \xi \in \mathbb{R}^{d};
\label{eq:Abound}
\end{equation}

\item[(ii)] $\mathbf{b}\in (L^{\infty}(\Omega))^{d}$, $\nabla\cdot\mathbf{b}\in L^{\infty}(\Omega)$, $c\in L^{\infty}(\Omega)$, $f\in L^{2}(\Omega)$.
\item[(iii)] Without loss of generality, we assume that $a_{\rm min} = 1$ and that
the diameter $D_{\Omega}$ of the domain $\Omega$ satisfies $D_{\Omega} \le
1$; otherwise the problem can simply be rescaled to achieve this.

\item[(iv)] In order to distinguish between strongly definite and strongly indefinite problems we allow    a splitting of $c$:
\begin{align} \label{eq:splitc} c = c^ + + c^-, \quad \text{with} \quad c^+(x) \geq 0 \quad \text{for all } x \in \Omega,  \quad \text{and} \quad c^- : = c- c^+.
\end{align}
\end{enumerate}
\end{assumption}
A natural choice would be to choose $c^+$ to be the non-negative  part of $c$, i.e.,
\begin{align*}
c^+(x) = \left\{ \begin{array}{ll}
c(x) \ &\text{when} \ c(x) \geq 0,\\
0 \ & \text{otherwise.} \end{array} \right.
\end{align*}
With this choice,  $c^- = 0$ if  $c$ is non-negative.
However other splittings are possible, e.g., $c^+ = 0$, and $c^- = c$.

Our analysis treats \eqref{eq:2-3} as a perturbation of the coercive   bilinear form
\begin{equation}\label{eq:2-4}
a(u,v) = \int_{\Omega}\left(A\nabla u\cdot{\nabla v} + c^+ u v\right) \,dx, \quad u, v \in H^1_0(\Omega).
\end{equation}
Using this we can write
\begin{align} \label{eq:b1}
b(u,v) = a(u,v) + (\mathbf{b}\cdot \nabla u , v) + (c^-u,v), \quad \text{for all} \quad u,v \in H^1_0(\Omega).
\end{align}
Alternatively,  after integration by parts, we can rewrite $b(u,v)$ as
\begin{equation}\label{eq:2-5}
b(u, v) = a(u,v) - (\mathbf{b}\cdot\nabla v,u)+(\tilde{c}u,v), \quad \text{for all} \quad u,v \in H^1_0(\Omega),
\end{equation}
with
\begin{align}
\label{eq:ctilde} \tilde{c} = c^--\nabla\cdot \mathbf{b},
\end{align}
where in \eqref{eq:ctilde} we make use of  Assumption \ref{ass:1}(ii). Both expressions for $b$ will be useful later.

\begin{notation} For any subdomain $\Omega' \subset \Omega$, we write $(\cdot, \cdot)_{\Omega'}$ to denote the $L^2(\Omega')$ inner product,  with induced norm $\Vert \cdot \Vert_{L^2(\Omega')}$. When $\Omega' = \Omega$ we write these more simply as $(\cdot, \cdot)$  and $\Vert \cdot \Vert$ respectively. Further, the $L^\infty$ norm on $\Omega$ is denoted $\Vert \cdot \Vert_\infty$.  For all other norms we indicate the space explicitly as  a subscript, e.g.,
$\Vert \cdot \Vert_{H^1_0(\Omega)}$.
We will make extensive use of the energy norm
\begin{equation}\label{eq:energy}
\Vert u\Vert_{a} :=\sqrt{a(u,u)},\quad {\rm for}\;\,{\rm all}\,\;u\in H_{0}^{1}(\Omega).
\end{equation}
\end{notation}

Furthermore, we assume unique solvability for \eqref{eq:2-2} and its adjoint:
\begin{assumption} \label{ass:2}
\begin{enumerate}
\item[(i)]
For any $f \in L^2(\Omega)$, we assume
the problem \eqref{eq:2-2} has a unique solution $u \in H^1_0(\Omega)$ and there exists a constant $\Cstab > 0 $ such that
$$
\Vert u \Vert_{a}\ \leq \Cstab \Vert f \Vert,  \quad \text{for all} \ f \in L^{2}(\Omega).
$$
\item[(ii)] For the adjoint problem: seek $w \in H^1_0(\Omega)$ such that   $b(v,w) = (f,v), \ v \in H^1_0(\Omega)$, we assume that,  for any $f \in L^2(\Omega)$,  there exists a unique solution $w \in H^1_0(\Omega)$  and a constant $\Cstab^*>0$ such that
\begin{align} \label{eq:Cstabstar}
\Vert w \Vert_{a} \ \leq \Cstab^* \Vert f \Vert,  \quad \text{for all} \ f \in L^{2}(\Omega).
\end{align}
\end{enumerate}
\end{assumption}

\noindent
Note that, in contrast to other works (e.g., \cite{cai}, \cite[Chapter
5]{brenner}), we make no \emph{a priori} regularity assumptions on $u$
beyond those implied by Assumptions \ref{ass:1} and
\ref{ass:2}. Thus, our only explicit assumption on $A$ is its uniform positive-definiteness \eqref{eq:Abound}.

Let $\mathcal{T}_{h}$ be a shape-regular triangulation of the domain
$\Omega$ consisting of triangles (2-d) or tetrahedra (3-d) of maximal diameter
$h$. Let $V^{h}\subset H_0^1(\Omega)$ be {any  conforming finite element (FE)} space. The Galerkin approximation of (\ref{eq:2-2}) seeks $u_{h}\in V^{h}$ such that
\begin{equation}\label{eq:2-7}
b(u_{h},v) = (f,v),\quad \text{for all} \quad  v \in V^{h}.
\end{equation}

Let  $n$ denote the dimension of $V^{h}$ and let $\{\phi_{i}\}^{n}_{i=1}$ be a basis for $V^{h}$. The linear system corresponding to  (\ref{eq:2-7}) is given by
\begin{equation}\label{eq:2-8}
\mathbf{B}\mathbf{u} = \mathbf{f},
\end{equation}
where the matrix $\mathbf{B}$  is defined by  $\mathbf{B}_{ij} = b(\phi_{i},\phi_{j})$  and $\mathbf{f}
= (f_{i}) \in \mathbb{R}^{n}$ with $f_{i} = (f, \phi_{i})$.
The solvability of (\ref{eq:2-7})
is guaranteed by the following result from  \cite[Theorem 2]{schatz-1}.
\begin{lemma}[Schatz \& Wang, 1996]\label{lem:s1}
Suppose Assumptions \ref{ass:1} and \ref{ass:2} hold. Then there exists $h_{0}>0$, such that for each $0<h<h_{0}$, (\ref{eq:2-7}) has a unique solution $u_{h}$. In addition, for any given $\varepsilon>0$, there exists an $h_{1}=h_{1}(\varepsilon)$, such that
\begin{equation}
\Vert u-u_{h}\Vert_{H^{1}(\Omega)}\leq \varepsilon \Vert f\Vert \, ,
\end{equation}
for $0<h<h_{1}$, where $u$ is the solution of (\ref{eq:2-2}).
The same holds true also for the FE solution to the adjoint
problem,  introduced in Assumption \ref{ass:2}(ii).
\end{lemma}
Note that Lemma \ref{lem:s1} implies that, for any $\varepsilon > 0$,
there exists $h_2 := h_1(\varepsilon/  (\sqrt{a_{\max} + \Vert c^+ \Vert_\infty}))$
such that
\begin{equation}
\Vert u-u_{h}\Vert_{a}\leq \varepsilon \Vert f\Vert\,,
\quad \text{for} \ \ 0<h<h_{2}\,.
\end{equation}

To finish this subsection, we introduce the following Friedrichs inequality with an explicit constant for subsequent use.
\begin{lemma}[Friedrichs inequality {\cite[Theorem 13.19]{Leoni}}]
\label{lem:2-4-0}Let $\Omega' \subset \mathbb{R}^{d}$ be an open set that lies between two parallel hyperplanes. Then, for all $u\in H_{0}^{1}(\Omega')$,
\begin{equation}
\label{eq:Fried} \Vert u \Vert _{L^{2}(\Omega')} \leq \frac{L}{\sqrt{2}} \Vert \nabla u \Vert_{L^{2}(\Omega')},
\end{equation}
where $L$ is the distance between the two hyperplanes.
\end{lemma}

Combining \eqref{eq:Fried} with Assumption \ref{ass:1}(iii), for any subdomain $\Omega' \subset \Omega$  with diameter $H$, we have the estimate \begin{equation}\label{eq:energy_est}
\Vert u\Vert_{L^{2}(\Omega')}\leq \frac{H}{\sqrt{2}}\Vert \nabla u\Vert_{L^{2}(\Omega')} \leq \frac{H}{\sqrt{2}}  \Vert u \Vert_{a, \Omega'},  \quad \text{for all} \quad  u\in H^{1}_{0}(\Omega').
\end{equation}

\subsection{Domain decomposition}
\label{subsec:dd}

In order to construct a one-level Schwarz preconditioner for (\ref{eq:2-8}), we first partition the domain $\Omega$ into a set of non-overlapping subdomains $\{\Omega^{\prime}_{i} \}_{i=1}^{N}$. We assume the boundaries of the $\Omega_i'$ are  resolved by $\mathcal{T}_{h}$. Each subdomain $\Omega^{\prime}_{i}$ is extended to a domain $\Omega_{i}$ by adding one or more layers of adjoining mesh elements, thus   creating  an overlapping decomposition $\{\Omega_{i} \}_{i=1}^{N}$ of $\Omega$. We assume that each $\Omega_{i}$ has a diameter $H_{i}$ and set  $H = \max_i\{ H_i\} $.

For each $i=1,\ldots,N$ we define
\begin{equation}\label{eq:2-9}
\widetilde{V}_{i} = \{v|_{\Omega_{i}}\;:\; v\in V^{h}\} \subset H^1(\Omega_i), \quad \text{and} \quad
V_{i} = \{v \in \widetilde{V}_i \; : \; v\vert_{\partial {\Omega_i}} = 0 \} \subset H_0^1(\Omega_i).
\end{equation}
For any $u$, $v\in \widetilde{V}_{i}$ and $i=1,\ldots,N$ we
define the local bilinear forms
\begin{align}
a_{\Omega_{i}}(u,v) :=&  \int_{\Omega_{i}}\left(A\nabla u\cdot{\nabla v}\, + c^+ u v\right) dx, 
\nonumber \\
b_{\Omega_{i}}(u,v) :=&  \int_{\Omega_{i}}\big(A \nabla u\cdot{\nabla v} + \mathbf{b}\cdot \nabla u \, v+ cu v\big)dx\nonumber   \\
=& \mathop{} a_{\Omega_{i}}(u,v) + (\mathbf{b} \cdot \nabla u, v)_{\Omega_i} + (c^-\, u , v)_{\Omega_i}.
\label{eq:2-10}
\end{align}
Then, in analogy to \eqref{eq:energy}, we write
\begin{equation}\label{eq:2-11}
\Vert u \Vert_{a,\Omega_{i}} = \sqrt{a_{\Omega_{i}}(u,u)},\quad {\rm for}\;\,{\rm all}\,\;u\in V_{i},
\end{equation}
which is a norm on $V_{i}$.

For any $v_{i}\in V_i$, let $R_i^\top v_i$ denote its zero extension to all of $\Omega$. Then, because of \eqref{eq:2-9},
\begin{equation}\label{eq:2-12}
R_{i}^{\top} \colon V_{i} \rightarrow V^{h},\quad i=1,\ldots,N.
\end{equation}
The $L^2(\Omega)$ adjoint of $R_i^\top$ is the
\emph{restriction operator} $R_{i} \colon V^h \rightarrow  V_i$,  defined by $R_iv = v\vert_{\Omega_i}$.
Using the extension operators, the bilinear forms in
\eqref{eq:2-10} can be defined equivalently by
\begin{equation}\label{eq:2-13}
b_{{\Omega}_{i}}(u,v) : = b(R_{i}^{\top}u,R_{i}^{\top}v), \quad
a_{{\Omega}_{i}}(u,v) : = a(R_{i}^{\top}u,R_{i}^{\top}v), \ \ \text{and} \ \ (u,v)_{\Omega_{i}} : = (R_{i}^{\top}u,R_{i}^{\top}v),
\end{equation}
for all $u,v\in V_{i}$.

In matrix form, the classical one-level additive Schwarz preconditioner for \eqref{eq:2-8} can now be defined as
\begin{equation}\label{eq:2-14}
\mathbf{M}_{AS,1}^{-1} = \sum_{i=1}^{N}\mathbf{R}_{i}^{T}\mathbf{B}_{i}^{-1}\mathbf{R}_{i},\quad \text{where} \quad \mathbf{B}_{i} := \mathbf{R}_{i}\mathbf{B}\mathbf{R}_{i}^{T}, \quad i = 1, \ldots, N.
\end{equation}
Here, $\mathbf{R}_{i}$ denotes the matrix representation of $R_{i}$
with respect to the basis $\{\phi_{i}\}^{n}_{i=1}$ of $V^h$.

It is well-known that, since each application of $\mathbf{M}_{AS,1}^{-1}$
exchanges information only between neighbouring subdomains, this preconditioner is generally not scalable
with respect to the number of subdomains $N$. To facilitate a global
exchange of information, a solve in a coarse space is usually added.
Let ${V_0} \subset V^{h}$ be such a coarse space; a classical example is an
$H^1_0(\Omega)$-conforming FE space corresponding to a coarse triangulation $\mathcal{T}_0$
of  $\Omega$, chosen so  that
$\mathcal{T}_h$ is a (regular) refinement of $\mathcal{T}_0$.
Let $R_{0}^{\top} \colon V_{0}\rightarrow V^{h}$ denote
the natural embedding (i.e., the identity operator)  and let  $R_{0}$ be the $L^2$ adjoint
of $R_0^\top$, i.e.,
$$ (v_0, R_0 w) := (R_0^\top v_0 , w)   \quad \text{for all} \quad w \in V^h, v_0 \in V_0. $$


In matrix form, the two-level additive Schwarz preconditioner can now simply be written as
\begin{equation}\label{eq:2-24}
\mathbf{M}_{AS,2}^{-1} = \sum_{i=0}^{N}\mathbf{R}_{i}^{T}\mathbf{B}_{i}^{-1}\mathbf{R}_{i}, \quad \text{where} \quad  \mathbf{B}_{i} := \mathbf{R}_{i}\mathbf{B}\mathbf{R}_{i}^{T}, \quad i = 0, \ldots, N,
\end{equation}
and where $\mathbf{R}_{0}$ is the matrix representation of $R_{0}$ with
respect to a specified basis of the coarse space $V_0$. The preconditioned system \eqref{eq:2-9} then reads
\begin{equation}\label{eq:2-23}
\mathbf{M}_{AS,2}^{-1} \mathbf{B} \mathbf{u} \ =\  \mathbf{M}_{AS,2}^{-1} \mathbf{f} .
\end{equation}

The convergence of GMRES for \eqref{eq:2-23} was analysed in
\cite{cai} for a classical piecewise linear coarse space
(see also \cite[Chapter~11]{toselli}). Under an a priori regularity
assumption on the solution $w$ of the adjoint problem (see the discussion just after  Assumption \ref{ass:2}(ii)),
and for sufficiently small coarse mesh
size~$H$, it was shown in~\cite{cai} that GMRES applied to \eqref{eq:2-23}  converges in a number of iterations
that is independent of~$h$ and depends only  on
$H/\delta$, where $\delta>0$ is the size of the overlap of the
subdomains. With the classical coarse space, the convergence rate of GMRES also depends in general
on the coefficients appearing in the PDE, in particular on $a_{\max}/a_{\min}$.
In order to achieve robustness to strong heterogeneities in the
coefficients and to avoid a priori assumptions on the regularity
of the adjoint solution, we will  introduce in Section~\ref{subsec:GenEO} an alternative coarse
space $V_0$.

Before leaving this subsection we introduce some notation which will be useful in the analysis later.
For each $i = 0, \ldots,N$ we define projections  $T_{i} \colon V^{h}\rightarrow V_{i}$, by
\begin{equation}\label{eq:2-20}
b_{\Omega_{i}}(T_{i}u, v) = b(u,R_{i}^{\top}v), \quad \text{for all} \quad v\in V_{i},
\end{equation}
where $\Omega_0 := \Omega$.
Conditions which ensure that the operators $T_i$ are  well-defined are presented in
Section~\ref{subsec:solvability}. Given the operators $T_i$, we can define
the operator $T \colon V^{h}\rightarrow V^{h}$ as
\begin{equation}\label{eq:2-21}
T:= \sum_{i=0}^{N}R_{i}^{\top}T_{i},
\end{equation}
which represents the preconditioner \eqref{eq:2-24} in the following sense.
\begin{proposition}
For any $u, v \in V^h$ with nodal vectors $\mathbf{u}, \mathbf{v} \in \mathbb{R}^n$,
we have \begin{align}\label{eq:equivalence}
\langle\mathbf{M}_{AS,2}^{-1} \mathbf{B} \mathbf{u}, \mathbf{v} \rangle_{\mathbf{A}} = (Tu,v)_a,
\end{align}
where $\langle \cdot, \cdot\rangle_{\mathbf{A}}$ denotes the inner product on $\mathbb{R}^n$ induced by the matrix
$\mathbf{A}_{i,j} = a(\phi_i, \phi_j).$
\end{proposition}

\subsection{GenEO coarse space}
\label{subsec:GenEO}
In this subsection we introduce the GenEO coarse space first proposed
in \cite{spillane:2014}.

For $1\leq j\leq N$, let $\tilde{n}_{j}$ be the
dimension of $\widetilde{V}_{j}$ and let
$\{\phi^{j}_{1}, 
\ldots,\phi^{j}_{\tilde{n}_{j}}\}$ be a
nodal basis of~$\widetilde{V}_{j}$. We first define the local partition of
unity operator.

\begin{definition}[Partition of unity]\label{def:2-0}
  Let $\,\tmop{dof}(\Omega_j)$ denote the internal degrees of freedom on subdomain $\Omega_j$. For any degree of freedom $i$, let $\mu_{i}$ denote the number of subdomains for which $i$ is an internal degree of freedom, i.e.,
\begin{equation}\label{eq:2-15}
\mu_{i}:=\#\{j\;:\;1\leq j\leq N,\;\;i\in \tmop{dof}(\Omega_j)\}.
\end{equation}
Then, for $1\leq j\leq N$, the local partition of unity operator $\Xi_{j} \colon \widetilde{V}_{j}\rightarrow V_{j}$ is defined by
\begin{equation}\label{eq:2-16}
\displaystyle \Xi_{j}(v) :=\sum_{i\in \tmop{dof}(\Omega_j)}\frac{1}{\mu_i} v_{i}\phi^{j}_{i},\quad {\rm for}\; {\rm all} \;\; v= \sum_{i=1}^{\tilde{n}_{j}}v_{i}\phi^{j}_{i} \in \widetilde{V}_{j}.
\end{equation}
\end{definition}

The operators $\Xi_{j}$ form a partition of unity in the following sense \cite{spillane:2014}:
\begin{equation}\label{eq:2-17}
\displaystyle \sum_{j=1}^{N} R_{j}^{\top} \Xi_{j}(v|_{\Omega_j}) = v,\quad {\rm for}\; {\rm all} \;\; v \in V^{h}.
\end{equation}
The local generalized eigenproblems that form the basis of the
  GenEO coarse space are now introduced.
\begin{definition}\label{def:2-1}
For each $j=1,\ldots,N$ we define the following generalized eigenvalue problem:
\begin{equation}\label{eq:2-18}
\text{Find} \quad (p, \lambda) \in {\tV}_j\backslash \{0\} \times \mathbb{R} \quad \text{such that} \quad      a_{\Omega_{j}}(p,v) = \lambda \,a_{\Omega_{j}}(\Xi_{j}(p), \Xi_{j}(v)),\quad {\rm for}\;\,{\rm all}\;\,v\in \widetilde{V}_{j},
\end{equation}
where $\Xi_{j}$ is the local partition of unity operator from Definition~\ref{def:2-0}.
\end{definition}

Note that there is some flexibility in the choice of the bilinear form
on the right hand side of \eqref{eq:2-18}. The one chosen here is not
the same as the one in the original paper \cite{spillane:2014}. However, it is
identical to the one considered in \cite[Chapter~7]{dolean:2015}. See also
\cite{peter} for an analysis of both choices as well as a further one.

\begin{definition}\label{def:2-2}
(GenEO coarse space). For each $j=1,\ldots,N$, let
$(p_{l}^{j})_{l=1}^{m_{j}}$ be the eigenfunctions corresponding to the
$m_{j}$ smallest eigenvalues in (\ref{eq:2-18}).  We then define the coarse space $V_0$ by
\begin{equation}\label{eq:2-19}
V_0 := {\rm span}\{R_{j}^{\top}\Xi_{j}(p_{l}^{j}) \colon l=1,\ldots,m_{j} \text{ and } j=1,\ldots,N\}.
\end{equation}
\end{definition}

In \cite{spillane:2014}, a two-level overlapping Schwarz preconditioner with
this GenEO coarse space was analysed for self-adjoint and coercive
problems, in particular for \eqref{eq:2-7} with $\mathbf{b} = 0$
and $c(x) \ge 0$. It was
shown there that a conjugate gradient (CG) iteration applied
to the preconditioned system \eqref{eq:2-21} converges in a number of
iterations that is independent of $h$, $H$, and $\delta$, as well as of
the variation of the coefficients $A(x)$ and $c(x)$, in
particular independent of $a_{\max}/a_{\min}$.

The result in the self-adjoint and coercive case holds without any
constraints on $H$ or $\Theta$. This cannot be expected in the
indefinite or non-self-adjoint case. However, by combining the
ideas from \cite{cai} and \cite{spillane:2014}, a coefficient-robust theory for
two-level overlapping Schwarz with GenEO coarse space can still
be developed, as we will see in the following two sections.

%% file: theoretical_tools.tex
\section{Theoretical tools}\label{sec-2}

\subsection{Basic properties of GenEO coarse space}
\label{subsec:basic}
In this subsection, we give some important properties of the GenEO coarse space for subsequent use. The following lemma gives an  error estimate for a local projection operator approximating a function $v\in \widetilde{V}_{j}$ in the space spanned   by the eigenfunctions in \eqref{eq:2-18}.
\begin{lemma}[{\cite[Lemma 2.11]{spillane:2014}}]\label{lem:3-1}
Let $j\in\{1,\ldots,N\}$ and $\{(p_{l}^{j},\lambda_{l}^{j})\}_{l=1}^{\dim(V_{j})}$ be the eigenpairs of the generalized eigenproblem \eqref{eq:2-18}. Suppose that $m_{j}\in\{1,\ldots, \dim(V_{j})-1\}$ is such that $0< \lambda_{m_{j}+1}^{j} < \infty$. Then the local projection operator $\Pi_{m_j}^{j}$,  defined by
\begin{equation}\label{eq:3-1}
\Pi_{m_j}^{j} v: = \sum_{l=1}^{m_j}a_{\Omega_{j}}(\Xi_{j}(v), \Xi_{j}(p_l^j))\,p_l^j ,
\end{equation}
satisfies the  estimates
\begin{align}
\Vert v - \Pi_{m_j}^{j} v \Vert_{a, \Omega_j} \ \leq \ \Vert v \Vert_{a, \Omega_j},
\quad \text{and}  \quad \big\Vert \Xi_{j}(v- \Pi_{m_j}^{j} v) \big\Vert_{a, \Omega_{j}}^{2} \leq \frac{1}{\lambda_{m_{j}+1}^{j}}\big\Vert v- \Pi_{m_j}^{j} v  \big\Vert_{a, \Omega_{j}}^{2},
\label{eq:3-2}
\end{align}
for all $v\in \widetilde{V}_{j}$.
\end{lemma}
Building on these local error estimates, we can now establish
a global approximation property for the GenEO coarse space. This property is
a crucial element of the  robustness proof for  the two-level Schwarz
method for the indefinite problem.
Before proceeding, we formally define the integer $k_0$, which measures
the number of subdomains any element can belong to, as
\begin{align}
k_{0} = \max_{\tau\in \mathcal{T}_{h}}\big(\#\{\Omega_{j}: 1\leq j\leq
N, \;\tau \subset \Omega_{j}\}\big).  \label{eq:defk0}
\end{align}
\begin{lemma}\label{lem:3-3}
Under the same conditions as Lemma \ref{lem:3-1},  let $v\in V^h$. Then
\begin{equation}\label{eq:3-5}
\inf_{z\in V_{0}}\Vert v-z\Vert^{2}_{a} \leq k_{0}^{2}\Theta \Vert v\Vert^{2}_{a} , \quad \text{where} \quad
\Theta = \left( \min_{1\leq j\leq N} \lambda_{m_{j}+1}^{j}\right)^{-1}.
\end{equation}
\end{lemma}
\begin{proof}
Given $v \in V^h$, we define
\begin{align}\label{eq:z0}    z_{0}:= \sum_{j=1}^{N} \Xi_{j}(\Pi_{m_j}^{j} v|_{\Omega_j}).
\end{align}
Then, making use of \eqref{eq:2-17}, we have
\begin{equation}\label{eq:3-6}
\Vert v-z_{0}\Vert^{2}_{a}= \left\Vert \sum_{j=1}^{N} \Xi_{j}(v|_{\Omega_j}- \Pi_{m_j}^{j} v|_{\Omega_j}) \right\Vert^{2}_{a} \leq k_{0} \sum_{j=1}^{N}\Vert\Xi_{j}(v\vert_{\Omega_j}- \Pi_{m_j}^{j} v_{\Omega_j}) \Vert^{2}_{a,\Omega_j},
\end{equation}
where in the last step we used a standard estimate, e.g., \cite[Equation (2.11)]{spillane:2014}. Then,  using Lemma~\ref{lem:3-1}, we have
\begin{equation}\label{eq:3-7}
\begin{split}
\Vert v-z_{0}\Vert^{2}_{a} &\leq k_{0} \sum_{j=1}^{N}\Vert\Xi_{j}(v- \Pi_{m_j}^{j} v) \Vert^{2}_{a,\Omega_j} \leq  k_{0} \sum_{j=1}^{N}\frac{1}{\lambda_{m_{j}+1}^{j}}\big\Vert v- \Pi_{m_j}^{j} v  \big\Vert_{a, \Omega_{j}}^{2} \\
&\leq k_{0} \sum_{j=1}^{N}\frac{1}{\lambda_{m_{j}+1}^{j}}\Vert v\Vert_{a, \Omega_{j}}^{2} \leq k_{0}^{2}\max_{1\leq j \leq N}\frac{1}{\lambda_{m_{j}+1}^{j}} \Vert v\Vert^{2}_{a},
\end{split}
\end{equation}
which implies \eqref{eq:3-5} as desired.
\end{proof}

The  following lemma can be proved by combining Lemma 2.9 of \cite{spillane:2014}
with Lemma \ref{lem:3-1} above.
It shows that  the
GenEO coarse space combined with the local finite element subspaces admit a
stable decomposition. Such a property is the usual tool for bounding the condition
number of the two-level Schwarz preconditioner in the classical positive
definite case. This property is also a key ingredient in this
work, since the indefinite problem is treated
as a perturbation of the associated positive definite problem in the convergence analysis.
\begin{lemma}[Stable decomposition]\label{lem:3-2}
Let $v\in V^h$, then the decomposition
\begin{equation}\label{eq:3-3}
z_{0}:= \sum_{j=1}^{N} \Xi_{j}(\Pi_{m_j}^{j} v|_{\Omega_j}), \quad z_{j} := \Xi_j(v|_{\Omega_j} - \Pi_{m_j}^{j} v|_{\Omega_j}),\quad {\rm for}\;\, j=1,\ldots,N,
\end{equation}
satisfies $\displaystyle v = \sum_{j=0}^{N}z_j$ and
\begin{equation}\label{eq:3-4}
\Vert z_{0}\Vert^{2}_{a} + \sum_{j=1}^{N}\Vert z_{j}\Vert^{2}_{a,\Omega_j}\leq \beta_{0}^{2}(k_0\Theta^{1/2})\Vert v\Vert_{a}^{2}, \quad \text{where} \quad   \beta_{0}(z)  = 2 \sqrt{1+z^2}, \quad \text{for} \ z \in \mathbb{R}^+ .
\end{equation}
\end{lemma}
\noindent {\bf Remark.} \   The result from \cite{spillane:2014} actually yields \eqref{eq:3-4} with $\beta_0^2(k_0\Theta^{1/2})$ replaced by $2 + k_0(2k_0+1)\Theta$. We give the estimate in the form  \eqref{eq:3-4} for simplicity,  showing that the key parameter is $k_0\Theta^{1/2}$.

As is well-known, the stable decomposition of Lemma \ref{lem:3-2} provides the theoretical basis for the analysis of preconditioners for systems arising from \eqref{eq:2-1} in the special case when $\mathbf{b}$ and $c$ vanish. This can be expressed by introducing the projection operators $P_{i} \colon V^h\rightarrow V^{i}$ such that
\begin{equation}\label{eq:4-4}
a_{\Omega_{i}}(P_{i}u, v) = a(u,R_{i}^{\top}v),\quad \forall v\in V^{i},\quad i=0,1,\ldots,N,
\end{equation}
These operators are well-defined, as shown in \cite[Section 2.2]{toselli}. Moreover, defining $P \colon V^h\rightarrow V^h$ as
\begin{equation}\label{eq:4-5}
P:=\sum_{i=0}^{N}R_{i}^{\top} P_{i},
\end{equation}
then using the stable decomposition property \eqref{eq:3-4}, it was proved in \cite[Section 2.3]{toselli} that any $u\in V^h$ satisfies
\begin{equation}\label{eq:4-6}
\beta_{0}^{-2}(k_0\Theta^{1/2})a(u,u)\leq a(Pu,u),
\end{equation}
and
\begin{equation}\label{eq:4-7}
\sum_{i=0}^{N}\Vert P_{i}u\Vert_{a,\Omega_{i}}^{2}\leq (k_{0}+1)\Vert u\Vert_{a}^{2}.
\end{equation}


%

\subsection{Solvability of subproblems}
\label{subsec:solvability}
In this subsection, we prove that the operators $T_i$ are well-defined by \eqref{eq:2-20}.
Stability properties of these operators are then analysed in Section \ref{sec:stability}.
The analysis builds again on the result of \cite{schatz-1}, given here in Lemma \ref{lem:s1}.
Since the underlying PDE is not coercive, this requires some discussion and
the imposition of some resolution conditions. Recalling items (ii) and (iv) from Assumption \ref{ass:1}, we can define
$$\cinf  = \mathrm{essinf}\{ c^-(x): \ x \, \in \Omega\}.$$
Note that, if $\cinf \geq 0$, then $c = c^- + c^+$ is almost everywhere non-negative.

For the well-posedness and stability of the local problems on each subdomain $\Omega_i$, we will need a coefficient-dependent resolution condition on $H$. We specify this via the function $C_0$, depending on $\mathbf{b}$ and $c^-$, given by
\begin{align} \label{eq:defC0new}
C_0(\mathbf{b}, c^-): = \left(\frac{1}{2} \Vert \mathbf{b} \Vert_\infty^2 \ +\ \max\{ - \cinf,0\}\right)^{1/2},
\end{align}
and note that $C_0$ vanishes when $\cinf = 0$ and $\mathbf{b} = \mathbf{0}$.

At various points in the following analysis we will use the standard inequality
\begin{align} \label{eq:standard}
pq \ \leq\  \frac{1}{2 \varepsilon} p^2 + \frac{\varepsilon}{2} q^2, \quad \text{valid for all} \quad  p,q \in \mathbb{R} \quad \text{and} \quad  \varepsilon > 0.
\end{align}

\begin{lemma}[$T_i$ is well-defined for each $i = 1,\ldots,N$] \label{lem:s2}
Suppose $H \, C_0(\mathbf{b}, c^-) < 1/2$. Then the
operators $T_{i}$, $i=1,\ldots,N$, are well-defined.
In particular, each $T_i$ is well-defined without a condition on $H$ when   $\cinf \geq  0$ and  $\mathbf{b}  = \mathbf{0}$.
\end{lemma}
\begin{proof}
We need to show that, for each $u\in V^h$ and $i=1,\ldots,N$, the discrete problem
\begin{equation}\label{eq:s1-11}
\text{Find} \quad   \phi_i \in V_i \quad \text{such that} \quad  b_{\Omega_{i}}(\phi_i, v) = b(u,R_{i}^{\top}v),\quad \forall v\in V_{i}
\end{equation}
has a unique solution $\phi_i$. Since \eqref{eq:s1-11} is equivalent to a square linear system, it is sufficient to prove that, when   $u=0$,  we only have the trivial solution. To do so, let $\phi_i\in V_i$ solve the problem
\begin{equation}\label{eq:s1-12}
b_{\Omega_{i}}(\phi_i, v) = 0,\quad \text{for all} \quad  v\in V_{i}.
\end{equation}
Then, taking $v = \phi_i$ in \eqref{eq:s1-12} and using the fact that
$-c^-(x) \leq \max \{-\cinf, 0\}$  for almost all $x \in \Omega$
and then \eqref{eq:energy}, we find
\begin{align}
\Vert {\phi_i}\Vert_{a,\Omega_{i}}^{2} &= -(\mathbf{b}\cdot\nabla {\phi_i}, {\phi_i})_{\Omega_{i}} - ({c^-}{\phi_i}, {\phi_i})_{\Omega_{i}} \nonumber \\
&\leq \Vert \mathbf{b} \Vert_{L^{\infty}(\Omega_{i})} \Vert \nabla {\phi_i} \Vert_{L^{2}(\Omega_i)} \Vert {\phi_i} \Vert_{L^{2}(\Omega_i)} + \max \{-\cinf, 0\} \Vert {\phi_i} \Vert^{2}_{L^{2}(\Omega_{i})} \nonumber \\
&\leq \frac{H}{\sqrt{2}}\Vert \mathbf{b} \Vert_\infty \Vert \nabla {\phi_i} \Vert^{2}_{L^{2}(\Omega_i)} + \frac{H^{2}}{2} \max \{-\cinf, 0\} \Vert \nabla {\phi_i} \Vert^{2}_{L^{2}(\Omega_{i})} \nonumber \\
&\leq \left(\frac{H}{\sqrt{2}}\Vert \mathbf{b}\Vert_\infty + \frac{H^{2}}{2}  \max \{-\cinf, 0\}\right)\Vert {\phi_i}\Vert_{a,\Omega_{i}}^{2}. \label{eq:phizero}
\end{align}
Now it is easy to see that the condition $H \, C_0(\mathbf{b}, c^-) < 1/2$ ensures that
$$ \frac{H}{\sqrt{2}} \Vert \mathbf{b} \Vert_\infty < \frac{1}{2} \quad \text{and} \quad \frac{H^2}{2} \max\{-\cinf, 0\} < 1/8.$$
Thus \eqref{eq:phizero} implies
$$\Vert \phi_i \Vert_{a,\Omega_{i}}^{2} \leq \frac{5}{8}  \Vert \phi_i\Vert_{a,\Omega_{i}}^{2}, $$
and it follows that $\phi_i = 0$.
\end{proof}
Now, to show that $T_0$ is also well-defined, we need to impose a resolution condition on $\Theta$, specified using the function $C_0(\mathbf{b}, c^-)$, together with another coefficient-dependent function
\begin{align}
C_{1}(\mathbf{b}, c^-) := 1 + \Vert \mathbf{b}\Vert_{\infty} +\Vert c^- \Vert_{\infty}. \label{eq:s1-8a}
\end{align}
Recall also the quantities $k_0$, defined in \eqref{eq:defk0}, and $\Cstab^{\ast}$, defined in Assumption \ref{ass:2}(ii).

\begin{lemma}[$T_0$ is well-defined] \label{lem:s2b}
Suppose that $$[(1+\Cstab^*) \, C_0(\mathbf{b},c^-)\,C_1(\mathbf{b},c^-) k_0 ] \Theta^{1/2} < 1/\sqrt{2}.$$ Then there exists $h_* >0$ such that, for all
$h\leq h_*$, the operator $T_{0}$ is well-defined. In particular, $T_0$ is well-defined, without
condition on $\Theta$, when $\cinf \geq  0$ and  $\mathbf{b}  = \mathbf{0}$.
\end{lemma}
\begin{proof}
Analogously to Lemma \ref{lem:s2},
it is sufficient to consider any solution $\phi$ to the problem
\begin{equation}\label{eq:s1-14}
\text{Find} \quad  \phi_0 \in V_0\quad  \text{such that} \quad  b({\phi_0}, v) = 0,\quad \text{for all} \quad  v\in V_{0},
\end{equation}
and then to  show that  ${\phi_0} = 0$.
To do so, given $\phi_0$, we first note that
\begin{align}
0 &= b(\phi_0, \phi_0) \ = \ a(\phi_0 , \phi_0) +  (c^-\phi_0, \phi_0) + (\mathbf{b} \cdot \nabla \phi_0, \phi_0) \nonumber \\
&\geq \Vert \phi_0 \Vert_a^2 + \cinf \Vert \phi_0 \Vert^2 + (\mathbf{b} \cdot \nabla \phi_0 , \phi_0) \nonumber \\
&\geq \Vert \phi_0 \Vert_a^2 + \cinf \Vert \phi_0 \Vert^2 -  \Vert \mathbf{b} \Vert_\infty \Vert  \nabla \phi_0 \Vert \, \Vert \phi_0 \Vert \nonumber \\
&\geq \Vert \phi_0 \Vert_a^2 + \cinf \Vert \phi_0 \Vert^2 - \frac{1}{2} \left( \Vert \phi_0\Vert_a^2 + \Vert \mathbf{b} \Vert_\infty ^2 \Vert \phi_0 \Vert^2 \right) \nonumber \\
&= \frac{1}{2} \Vert \phi_0 \Vert_a^2  - \left( \frac{1}{2} \Vert \mathbf{b} \Vert_\infty^2 - \cinf\right) \Vert \phi_0 \Vert^2 \nonumber \\
&\geq \frac{1}{2} \Vert \phi_0\Vert_a^2 - C_0(\mathbf{b},c^-)^2 \Vert \phi_0\Vert^2, \label{eq:zerogeq}
\end{align}
where in the third step we used \eqref{eq:standard} and \eqref{eq:energy}, and in the final step we used the definition of $C_0(\mathbf{b},c^-)$ in \eqref{eq:defC0new}.

To finish the proof we use a duality argument to estimate
$\Vert \phi_0 \Vert$ from above, in terms of $\Vert \phi_0 \Vert_a$. As such, let $w$ solve the auxiliary problem:
\begin{equation}\label{eq:s1-15}
\text{Find} \quad  w \in H_0^{1}(\Omega) \quad
\text{such that} \quad  b(v,{w}) = ({\phi_0}, v),\quad \text{for all} \ v\in H_{0}^{1}(\Omega).
\end{equation}
Further, let $w_h \in V^h$ denote its finite element approximation, i.e.,
\begin{equation}
\label{eq:FEM}
b(v,{w_{h}}) = ({\phi_0}, v),\quad \text{for all} \quad  v\in V^h.
\end{equation}

Applying Lemma~\ref{lem:s1} to the adjoint problem \eqref{eq:s1-15},
we see that there exists an $h_\ast >0$ such that, for all $0<h<h_\ast$,
\begin{equation}\label{eq:s1-16}
\Vert  w-w_{h} \Vert_{a} \leq \Vert {\phi_0} \Vert\,.
\end{equation}
(See also the discussion following Lemma~\ref{lem:s1}.) Now, combining
\eqref{eq:s1-16} and  Assumption \ref{ass:2}(ii), it follows that
\begin{equation}\label{eq:s1-16-0}
\Vert w_{h}\Vert_{a}\leq (\Cstab^{\ast}+1)\Vert {\phi_0} \Vert.
\end{equation}
Since $V_0 \subset V^h$, we can substitute  $v = {\phi_0}$ in \eqref{eq:FEM} and then use
\eqref{eq:s1-14} and \eqref{eq:energy} to obtain, for any ${z} \in
V_{0}$,
\begin{align}
\Vert {\phi_0}\Vert^{2} &= b({\phi_0},{w}_{h}) = b({\phi_0}, {w_h - z}) \nonumber \\
&\leq\Vert {\phi_0} \Vert_{a} \Vert {w_h - z}\Vert_{a} + \Vert \mathbf{b}\Vert_\infty \Vert \nabla \phi_0\Vert \Vert w_h - z\Vert + \Vert c^- \Vert_\infty \Vert \phi_0\Vert \Vert w_h - z\Vert \nonumber \\
&\leq \left(1+\Vert \mathbf{b}\Vert_\infty D_{\Omega} + \Vert c^- \Vert_\infty D^{2}_{\Omega}\right) \Vert {\phi_0} \Vert_{a} \Vert {w_h - z}\Vert_{a}. \label{eq:s1-17}
\end{align}
Thus, using Assumption \ref{ass:1}(iii), we have
\begin{equation}\label{eq:s1-17-0}
\Vert {\phi_0}\Vert^{2} \leq C_{1}(\mathbf{b}, c^-)\Vert {\phi_0} \Vert_{a} \Vert {w_h - z}\Vert_{a}, \quad\text{for all} \quad z \in V_0.
\end{equation}

Now, from \eqref{eq:s1-17-0} using Lemma~\ref{lem:3-3} and then \eqref{eq:s1-16-0}, we further have
\begin{align}
\Vert {\phi_0} \Vert^{2} &\leq {C_{1}(\mathbf{b}, c^-)} \Vert {\phi_0}\Vert_{a} \inf_{{z} \in V_{0}}\Vert {w_{h}- z} \Vert_{a} \leq {C_{1}(\mathbf{b}, c^-) k_{0}} \Theta^{\frac{1}{2}} \Vert {\phi_0}\Vert_{a} \Vert {w}_{h}\Vert_{a} \nonumber \\
&\leq (1+\Cstab^*) C_{1}(\mathbf{b},c^-) k_0 \Theta^{1/2} \Vert {\phi_0}\Vert_{a} \Vert {\phi_0}\Vert, \label{eq:s1-18}
\end{align}
and thus
\begin{equation}\label{eq:s1-19}
\Vert {\phi_0} \Vert \leq (1+ \Cstab^*) C_{1}(\mathbf{b},c^-)k_0 \Theta^{1/2} \Vert {\phi_0}\Vert_{a}.
\end{equation}
Finally, inserting \eqref{eq:s1-19} into \eqref{eq:zerogeq}, we obtain
\begin{align*}
0 &\geq \left(1/2 - K^2 \right) \Vert \phi_0 \Vert_a^2,
\end{align*}
where
\begin{align} \label{eq:defbeta}
K := (1+\Cstab^*) \, C_0(\mathbf{b},c^-)\,C_1(\mathbf{b},c^-) k_0 \Theta^{1/2},
\end{align}
and $\phi_0 = 0$ follows necessarily when $K < 1/\sqrt{2}$.
\end{proof}

\subsection{Stability estimates for $T_{i}$}\label{sec:stability}
In this subsection, we prove the stability estimates for the
operators $T_{i}$ that will be required in the following section to
establish the robustness of the two-level additive Schwarz method.

\begin{lemma}[Stability for $T_i$, $i = 1,\ldots,N$] \label{lem:4-0}
Suppose $H \, C_0(\mathbf{b}, c^-) < 1/4$. Then for all $u \in V^h$ and all $i \in \{ 1,\ldots,N\}$
\begin{equation}\label{eq:4-9}
\Vert T_{i}u \Vert_{a, \Omega_{i}} \le 2 \sqrt{2} \Vert u\Vert_{a,\Omega_{i}} +  2 \sqrt{\frac{2}{3}} H  \Vert c^- \Vert_{L^\infty(\Omega_i)}\, \Vert u\Vert_{L^2(\Omega_{i})}.
\end{equation}
\end{lemma}
\begin{proof}
Let $i\in\{1,\ldots,N\}$. Using the fact that
$b_{\Omega_{i}}(T_{i}u,T_{i}u)=b_{\Omega_{i}}(u, T_{i}u)$, we find
\begin{align}
\Vert T_{i}u \Vert_{a, \Omega_{i}}^{2} &=  a(T_iu, T_iu) \nonumber \\
&= b_{\Omega_{i}}(T_iu, T_{i}u) -\left[ (\mathbf{b}\cdot\nabla T_{i}u, T_{i}u) + (c^-T_{i}u, T_{i}u)_{\Omega_{i}} \right] \nonumber \\
&= b_{\Omega_{i}}(u, T_{i}u) -\left[(\mathbf{b}\cdot\nabla T_{i}u, T_{i}u) + (c^-T_{i}u, T_{i}u)_{\Omega_{i}}\right] \nonumber \\
\label{eq:Tiu}
\begin{split}
& =a_{\Omega_i}(u,T_i u) + (\mathbf{b}\cdot \nabla u, T_i u)_{\Omega_{i}} + (c^-u, T_i u)_{\Omega_{i}} \\
& \qquad - \left[ (\mathbf{b}\cdot\nabla T_{i}u, T_{i}u)_{\Omega_{i}} + (c^-T_{i}u, T_{i}u)_{\Omega_{i}}\right].
\end{split}
\end{align}
We estimate each of the terms in \eqref{eq:Tiu} as follows, using \eqref{eq:energy} and \eqref{eq:standard}, to obtain
\begin{subequations}
\label{eq:t}
\begin{align} \label{eq:t1}
a_{\Omega_i} (u,T_i u) &\leq \frac{1}{4} \Vert T_i u \Vert_{a, \Omega_i} ^2 + \Vert u \Vert_{a, \Omega_i}^2, \\
(\mathbf{b} \cdot \nabla u , T_i u) & \leq \Vert \mathbf{b} \Vert_{L^\infty(\Omega_i)} \Vert u \Vert_{a,\Omega_i} \Vert T_i u \Vert_{L^2(\Omega_i)} \leq \frac{H}{\sqrt{2}} \Vert \mathbf{b} \Vert_\infty  \Vert T_i u \Vert_{a, \Omega_i}\Vert u \Vert_{a,\Omega_i} \nonumber \\
&\leq \frac{H^2}{4} \Vert \mathbf{b} \Vert_{L^\infty(\Omega_i)}^2 \Vert T_i u \Vert_{a, \Omega_i}^2 + \frac{1}{2} \Vert u \Vert_{a, \Omega_i}^2, \label{eq:t2} \\
\text{and} \quad (c^-u, T_iu)_{\Omega_i} & \leq \frac{H}{\sqrt{2}} \Vert T_i u \Vert_{a, \Omega_i} \, \Vert c^- \Vert_{L^\infty(\Omega_i)} \Vert u \Vert_{L^2(\Omega_i)} \nonumber \\
& \leq \frac{1}{4} \Vert T_i u \Vert_{a, \Omega_i}^2 + \frac{H^2}{2} \Vert c^-\Vert_{L^\infty(\Omega_i)}^2 \Vert u \Vert _{L^2(\Omega_i)}^2. \label{eq:t3}
\end{align}
\end{subequations}
Combining \eqref{eq:t} with \eqref{eq:Tiu} and estimating the final term in \eqref{eq:Tiu} using \eqref{eq:phizero}, we obtain
\begin{align*}
\Vert T_i u \Vert_{a, \Omega_i}^2 &\leq \left(\frac{1}{2} + \frac{H^2}{4} \Vert \mathbf{b} \Vert_{L^\infty(\Omega_i)}^2  \right) \Vert T_i u \Vert_{a, \Omega_i}^2 \\
& \qquad + \left( \frac{3}{2} \Vert u \Vert_{a, \Omega_i} ^2 + \frac{H^2}{2} \Vert c^- \Vert_{L^\infty(\Omega_i)}^2 \Vert u \Vert_{L^2(\Omega_i)}^2 \right) \\
& \qquad + \left(\frac{H}{\sqrt{2}} \Vert \mathbf{b} \Vert_{L^\infty(\Omega_i)} + \frac{H^2}{2} \max \{ -\cinf, 0\} \right) \Vert T_i u \Vert_{a, \Omega_i}^2.
\end{align*}
Hence, we deduce that
\begin{align}
\label{eq:stabTi}
\begin{split}
\frac{1}{2}  \Vert T_i u \Vert_{a, \Omega_i}^2 & \leq \left[\frac{H}{\sqrt{2}}\Vert \mathbf{b} \Vert_\infty  + \frac{H^2}{4} \Vert \mathbf{b} \Vert_\infty^2 + \frac{H^2}{2} \max\{ -\cinf, 0\} \right]  \Vert T_i u \Vert_{a, \Omega_i}^2 \\
& \qquad + \left(\frac{3}{2} \Vert u \Vert_{a, \Omega_i} ^2 + \frac{H^2}{2} \Vert c^- \Vert_{L^\infty(\Omega_i)}^2   \Vert u \Vert_{L^2(\Omega_i)}^2 \right).
\end{split}
\end{align}
Arguing as in the proof of Lemma \ref{lem:s2}, it is now easy to see that the assumption $H \, C_0(\mathbf{b}, c^-) < 1/4$ ensures that the term in the square brackets in \eqref{eq:stabTi} is bounded above by $5/16$. Thus we have
\begin{align*}
\frac{3}{16}  \Vert T_i u \Vert_{a, \Omega_i}^2 \leq   \left(\frac{3}{2} \Vert u \Vert_{a, \Omega_i} ^2 + \frac{H^2}{2} \Vert c^- \Vert_{L^\infty(\Omega_i)}^2   \Vert u \Vert_{L^2(\Omega_i)}^2 \right),
\end{align*}
from which the result \eqref{eq:4-9} follows.
\end{proof}

We now examine the stability of $T_0$ in the following lemma, for which we need another constant, namely
\begin{align} \label{eq:defC2}
C_2(\mathbf{b}, c^-) := 1 + \Vert \mathbf{b} \Vert_{\infty} + \Vert \nabla \cdot \mathbf{b} \Vert_{\infty} + \Vert {c^-} \Vert_{\infty}.
\end{align}

\begin{lemma}[Stability for $T_0$]\label{lem:4-2}
Suppose that
\begin{align}
\label{eq:4-0}
(1+ \Cstab^*) \, C_0(\mathbf{b},c^-) \, C_1(\mathbf{b},c^-) \, k_0 \Theta^{1/2} < \frac{1}{\sqrt{2}}.
\end{align}
Then there exists $h_*>0$ such that, for $h \leq h_*$,
\begin{equation}\label{eq:4-1}
\Vert T_{0}u - u\Vert \leq (1+\Cstab^*) \, C_{1}(\mathbf{b}, c^-) \, k_0 \Theta^{1/2} \, \Vert T_{0}u - u\Vert_{a}, \quad \text{for all} \quad u\in V^h.
\end{equation}
Suppose, in addition, that
\begin{align} \label{eq:addhyp}
(1+\Cstab^*) \, C_1(\mathbf{b}, c^-) \, C_2(\mathbf{b},c^-) \, k_0 \Theta^{1/2} < \frac{1}{2}.
\end{align}
Then
\begin{equation}\label{eq:4-11}
\Vert u - T_{0}u \Vert_{a} \le \sqrt{2} \Vert u \Vert_{a}, \quad \text{for all} \quad u\in V^h.
\end{equation}
\end{lemma}
\begin{proof}
Under condition \eqref{eq:4-0},
Lemma \ref{lem:s2b} ensures the existence of $h_*$ such that, when $h \leq h_*$,
$T_0 \colon V^h \rightarrow V_0$ is well-defined. We then consider the auxiliary problem:
\begin{equation}\label{eq:4-2}
\text{Find} \quad  w_h \in V^h \quad \text{such that} \quad b(v,w_{h}) = (T_{0}u-u, v),\quad \text{for all} \quad  v\in V^h.
\end{equation}
Then, since $b(T_0u - u,z) = 0 $ for all $z \in V_0$, \eqref{eq:4-1} follows by the same argument as used in the proof of \eqref{eq:s1-19}.

By the definition of $P_0$, we have $a(T_{0}u, u-P_{0}u)= 0$. Also, since $P_{0}u-T_{0}u \in V^{0}$, we have
$b(u-T_{0}u, P_{0}u-T_{0}u)= 0$. Thus, using these relations and \eqref{eq:2-5} (twice), we can write
\begin{align*}
\Vert u - T_{0}u\Vert_{a}^{2} &= b(u-T_{0}u, u-P_{0}u) + \big(\mathbf{b}\cdot \nabla(u-T_{0}u),\,u-T_{0}u\big)-(\tilde{c}(u-T_{0}u), u-T_{0}u) \\
&= a(u-T_{0}u, u-P_{0}u) + \big(\mathbf{b}\cdot\nabla(P_{0}u-T_{0}u),\,u-T_{0}u\big)+(\tilde{c}(u-T_{0}u), T_{0}u-P_{0}u) \\
&= a(u, u-P_{0}u) + \big(\mathbf{b}\cdot\nabla(P_{0}u-T_{0}u),\,u-T_{0}u\big)+(\tilde{c}(u-T_{0}u), T_{0}u-P_{0}u).
\end{align*}
Hence, using \eqref{eq:energy}, it follows that
\begin{align}
\label{eq:4-12}
\begin{split}
\Vert u - T_{0}u\Vert_{a}^{2} &\leq \Vert u\Vert_{a} \Vert u - P_{0}u \Vert_{a} + \Vert \mathbf{b}\Vert_{\infty}\Vert u - T_{0}u \Vert  \Vert P_{0}u - T_{0}u \Vert_{a} \\
& \qquad + \Vert \tilde{c}\Vert_{\infty}\Vert u - T_{0}u \Vert \Vert P_{0}u - T_{0}u \Vert.
\end{split}
\end{align}

Now, by the definition of the projection $P_{0}$, we have
\begin{equation}\label{eq:4-13}
\Vert u - P_{0}u \Vert_{a} \leq \Vert u\Vert_{a},\quad \text{and} \quad \Vert P_0 u - T_0 u \Vert \leq \Vert P_{0}u - T_{0}u \Vert_{a} = \Vert P_0(u - T_0u)\Vert_a \leq \Vert u - T_{0}u \Vert_{a}.
\end{equation}
Hence, using \eqref{eq:4-1} and \eqref{eq:4-13} in \eqref{eq:4-12}, we obtain
\begin{align*}
\Vert u - T_{0}u\Vert_{a}^{2} &\leq \Vert u \Vert_{a}^{2} + \left( \Vert \mathbf{b}\Vert_{\infty} + \Vert \tilde{c} \Vert_{\infty} \right) \Vert u - T_0u\Vert_a \, \Vert u - T_0 u\Vert \\
&\leq \Vert u \Vert_{a}^{2} + (1+\Cstab^*) \, C_1(\mathbf{b},c^-) \, C_2(\mathbf{b}, c^-) \, k_0 \Theta^{1/2} \, \Vert u - T_0u\Vert_a^2,
\end{align*}
and the result follows.
\end{proof}

%% file: main_proof.tex
\section{Main results}
\label{Proof}

In this section we now state and prove the main theoretical result of this
paper, as well as a corollary, on the robust GMRES convergence on the
preconditioned system when using the GenEO coarse space, under
certain conditions on the size of $\mathbf{b}$ and $c$. We first recall some important quantities defined in the preceding section.
\begin{align}
\label{eq:constants}
\begin{split}
k_{0} &= \max_{\tau\in\mathcal{T}_{h}} \left(\#\{\Omega_{j} : 1\leq j\leq N, \; \tau \subset \Omega_{j}\}\right), \\
\Theta &= \left( \min_{1\leq j\leq N} \lambda_{m_{j}+1}^{j} \right)^{-1}, \\
C_{0}(\mathbf{b}, c^-) &= \left( \frac{1}{2} \Vert \mathbf{b} \Vert_\infty^2 + \max\{ -\cinf,0 \} \right)^{1/2}, \\
C_{1}(\mathbf{b}, c^-) &= 1 + \Vert \mathbf{b} \Vert_{\infty} + \Vert c^- \Vert_{\infty}, \\
C_{2}(\mathbf{b}, c^-) &= 1 + \Vert \mathbf{b} \Vert_{\infty} + \Vert \nabla \cdot \mathbf{b} \Vert_{\infty} + \Vert c^- \Vert_{\infty}, \\
\beta_{0}(z) &= 2 \sqrt{1+z^2}.
\end{split}
\end{align}
\begin{theorem}\label{thm:2-1}
Assume that $h\leq h_*$, where $h_*$ is as defined in Lemma \ref{lem:s2b}. Suppose $H$ and $\Theta$ are chosen so that
\begin{equation}\label{eq:2-28}
s = s(\Theta) := 2 \sqrt{2} (1+\Cstab^*) \, C_1(\mathbf{b}, c^-) \, C_2(\mathbf{b}, c^-) \, k_0^{3/2} \, \beta^{2}_{0}(k_0\Theta^{1/2}) \, \Theta^{1/2} < 1,
\end{equation}
and
\begin{equation}\label{eq:2-29}
t = t(H,\Theta) := 16 H \, C_1(\mathbf{b}, c^-) \, k_{0} \, \beta^{2}_{0}(k_0\Theta^{1/2}) < 1.
\end{equation}
Then, for all $u\in V_{h}$,
\begin{equation}\label{eq:2-26}
c_{1}(H,\Theta) \, a(u,u) \leq a(Tu,u),
\end{equation}
and
\begin{equation}\label{eq:2-27}
a(Tu,Tu) \leq c_{2}(H,\Theta) \, a(u,u),
\end{equation}
where $c_{1}(H,\Theta)$ and $c_{2}(H,\Theta)$ are given by
\begin{align}\label{eq:2-30}
\begin{split}
c_{1}(H,\Theta) &= \beta_{0}^{-2}(k_0\Theta^{1/2})(1-\max\{t,\,s\}), \\
c_{2}(H,\Theta) &= 12 + 32 k_{0}^{2}.
\end{split}
\end{align}
\end{theorem}

\begin{proof}
We begin by showing that the assumptions \eqref{eq:2-28} and \eqref{eq:2-29} ensure
that the hypotheses of Lemmas \ref{lem:4-0} and \ref{lem:4-2} hold.
First, since $k_0 \geq 1$ and $\beta_0^2(k_0\Theta^{1/2}) \geq 2$, \eqref{eq:2-28} implies
\begin{align} \label{eq:4a}
(1+ \Cstab^*) \,  C_1(\mathbf{b},c^- ) \, C_2(\mathbf{b},c^-) \, k_0 \Theta^{1/2} < 1/(4\sqrt{2}).
\end{align}
Further, by definition \eqref{eq:defC0new} and using \eqref{eq:standard} we have
\begin{align*}
C_0^2(\mathbf{b},c^-) &\leq \frac{1}{2} \Vert \mathbf{b} \Vert_{\infty}^2 + \Vert c^- \Vert_{\infty} \leq \frac{1}{2} \left( \Vert \mathbf{b} \Vert_{\infty}^2 + \Vert c^- \Vert_{\infty}^2 + 1 \right) \\
&\leq \frac{1}{2} \left( \Vert \mathbf{b} \Vert_{\infty} + \Vert c^- \Vert_{\infty} + 1 \right)^2 \leq \frac{1}{2} C_2^2(\mathbf{b}, c^-),
\end{align*}
and so \eqref{eq:4a} implies $(1+ \Cstab^*) \, C_1(\mathbf{b},c^-) \, C_0(\mathbf{b},c^-) \, k_0 \Theta^{1/2} < 1/8$, and
both hypotheses \eqref{eq:4-0} and \eqref{eq:addhyp} of Lemma \ref{lem:4-2} are satisfied.

Likewise, hypothesis \eqref{eq:2-29} implies $H \, C_1(\mathbf{b},c^-) < 1/32$ and so we have
\begin{align} \label{eq:max}
\max\{ H \Vert \mathbf{b}\Vert_{\infty}, H \Vert c^-\Vert_{\infty} \} < 1/32.
\end{align}
Combining \eqref{eq:max} with a straightforward calculation and using $H \leq \mathrm{diam}(\Omega) \leq 1$, shows that the hypothesis of Lemma \ref{lem:4-0} holds. In the remainder of the proof we use the results of Lemmas \ref{lem:4-0} and \ref{lem:4-2} without further justification.

Let $u\in V_{h}$, we first prove \eqref{eq:2-26}.
Using \eqref{eq:4-5},\eqref{eq:4-6}, \eqref{eq:b1} and the definition of the projections $T_{i}$, we obtain
\begin{align*}
& \beta_{0}^{-2}(k_0\Theta^{1/2}) a(u,u) \leq \sum_{i=0}^{N}a(u, R_{i}^{ \top}P_{i}u) \\
& \mbox{\hspace{1cm}} = \sum_{i=0}^{N}b(u, R^{ \top}_{i}P_{i}u) - \sum_{i=0}^{N}\big[\big(\mathbf{b}\cdot\nabla u, R^{\top}_{i}P_{i}u\big) + \big(c^-u, R^{\top}_{i}P_{i}u\big)\big]  \\
& \mbox{\hspace{1cm}} = \sum_{i=0}^{N}b(R_{i}^{\top}T_{i}u, R^{\top}_{i}P_{i}u) -\sum_{i=0}^{N}\big[\big(\mathbf{b}\cdot\nabla u, R^{\top}_{i}P_{i}u\big) + \big(c^-u, R^{\top}_{i}P_{i}u\big)\big]  \\[3mm]
& \mbox{\hspace{1cm}} = \sum_{i=0}^{N}a(R_{i}^{\top}T_{i}u, R^{\top}_{i}P_{i}u) + \sum_{i=0}^{N}\big[\big(\mathbf{b}\cdot\nabla(R_{i}^{\top}T_{i}u-u), \,R^{\top}_{i}P_{i}u\big) + \big(c^-(R_{i}^{\top}T_{i}u-u), R^{\top}_{i}P_{i}u\big)\big].
\end{align*}
Now, since $$ a(R_i^\top T_i u, R_i^\top P_i u) = a_{\Omega_i} ( T_i u, P_i u) =  a_{\Omega_i} ( T_i u, u) =  a ( R_i^\top T_i u, u), $$
we have, by definition \eqref{eq:2-21} of $T$,
\begin{align}
\label{eq:4-17}
\begin{split}
a(u,u) &\leq \beta_0^{2}(k_0\Theta^{1/2}) \, a(Tu, u) \\
& \quad + \beta_0^{2}(k_0\Theta^{1/2})\sum_{i=0}^{N}\big[\big(\mathbf{b}\cdot\nabla(R_{i}^{\top}T_{i}u-u), \,R^{\top}_{i}P_{i}u\big) + \big(c^-(R_{i}^{\top}T_{i}u-u), R^{\top}_{i}P_{i}u\big)\big].
\end{split}
\end{align}
We proceed by bounding the sum in \eqref{eq:4-17} in terms of $a(u,u)$.

First, we consider the summand in \eqref{eq:4-17} corresponding to $i=0$. Integrating by parts as in \eqref{eq:2-5}, recalling that $R_0^\top$ is the identity operator, and using \eqref{eq:energy_est}, we obtain
\begin{align*}
\big(\mathbf{b}\cdot\nabla &(R_{0}^{\top}T_{0}u-u), R^{\top}_{0}P_{0}u\big) + \big(c^-(R_{0}^{\top}T_{0}u-u), R^{\top}_{0}P_{0}u\big) \\
&= -\big(\mathbf{b}\cdot\nabla (R^{\top}_{0}P_{0}u), R_{0}^{\top}T_{0}u-u\big) + \big(\tilde{c}(R_{0}^{\top}T_{0}u-u), R^{\top}_{0}P_{0}u\big) \\
&\leq \Vert \mathbf{b}\Vert_{\infty}\Vert u-T_{0}u \Vert \Vert P_{0}u \Vert_{a} + \Vert \tilde{c}\Vert_{\infty}\Vert u-T_{0}u \Vert \Vert P_{0}u \Vert \\
&\leq C_{2}(\mathbf{b}, c^-) \Vert u-T_{0}u\Vert \Vert P_{0}u \Vert_{a} \\
&\leq (1+ \Cstab^*) \, C_{2}(\mathbf{b}, c^-) \, C_1(\mathbf{b},c^-) \, k_0 \Theta^{1/2} \Vert u-T_{0}u\Vert_a \Vert P_{0}u \Vert_{a} \\
&\leq \sqrt{2} (1 + \Cstab^*) \, C_{2}(\mathbf{b}, c^-) \, C_1(\mathbf{b},c^-) \, k_0 \Theta^{1/2} \Vert u\Vert_a \Vert P_{0}u \Vert_{a},
\end{align*}
where the last three estimates are obtained using \eqref{eq:defC2}, \eqref{eq:4-1}, and \eqref{eq:4-11}. Using this, together with the definition of $s$ in assumption \eqref{eq:2-28} and recalling $k_0 \geq 1$, we obtain
\begin{align}
\beta^{2}_{0}(k_0\Theta^{1/2})&\big[\big(\mathbf{b}\cdot\nabla (R_{0}^{\top}T_{0}u-u), R^{\top}_{0}P_{0}u\big) + \big(c^-(R_{0}^{\top}T_{0}u-u), R^{\top}_{0}P_{0}u\big)\big] \nonumber \\
&\leq \sqrt{2} \, C_{2}(\mathbf{b}, c^-) \, C_1(\mathbf{b},c^-) \, \beta^{2}_{0}(k_0\Theta^{1/2}) \, (1+ \Cstab^*) \, k_0 \Theta^{1/2} \Vert u \Vert_{a}\Vert P_{0}u \Vert_{a} \nonumber \\
&= \frac{s}{2k_{0}^{1/2}}\Vert u \Vert_{a} \Vert P_{0}u \Vert_{a}
\leq \frac{s}{\sqrt{2}(k_{0}+1)^{1/2}}\Vert u \Vert_{a} \Vert P_{0}u \Vert_{a}. \label{eq:4-20}
\end{align}

Next, we consider the summands in \eqref{eq:4-17} corresponding to $i=1,\ldots,N$. Using \eqref{eq:energy_est} and the fact that $R_i^\top$ corresponds to the zero extension from $\Omega_i$, we obtain
\begin{align}
\big(\mathbf{b}\cdot\nabla &(R_{i}^{ \top}T_{i}u-u), R^{\top}_{i}P_{i}u\big) + \big(c^-(R_{i}^{\top}T_{i}u-u), R^{\top}_{i}P_{i}u\big) \nonumber \\
&\leq \left( \Vert \mathbf{b}\Vert_{L^{\infty}(\Omega_{i})}\Vert u-T_{i}u \Vert_{a,\Omega_{i}} + \Vert c^-\Vert_{L^{\infty}(\Omega_{i})} \Vert u - T_{i}u \Vert_{L^{2}(\Omega_{i})}\right) \Vert P_{i}u \Vert_{L^{2}(\Omega_{i})}.
\label{eq:4-21}
\end{align}
Furthermore, by \eqref{eq:4-9} and \eqref{eq:energy_est} again and noting that \eqref{eq:max} implies $H \Vert c^{-} \Vert_{L^\infty(\Omega_i)} \leq 1$, we have
\begin{align*}
\Vert u - T_i u \Vert_{a, \Omega_i} \leq 4 \Vert u \Vert_{a, \Omega_i} + 2\Vert u \Vert_{L^2(\Omega_i)} \quad \text{and} \quad \Vert u - T_i u \Vert_{L^2(\Omega_i)} \leq 2\Vert u\Vert_{L^2(\Omega_i)} + \sqrt{3} \Vert u \Vert_{a, \Omega_i}.
\end{align*}
Inserting these estimates into \eqref{eq:4-21}, we obtain
\begin{align}
\big(\mathbf{b}\cdot\nabla &(R_{i}^{\top}T_{i}u-u), R^{\top}_{i}P_{i}u\big) + \big(c^-(R_{i}^{\top}T_{i}u-u), R^{\top}_{i}P_{i}u\big) \nonumber \\
&\leq \left( \Vert \mathbf{b}\Vert_{L^{\infty}(\Omega_{i})} + \Vert c^-\Vert_{L^{\infty}(\Omega_{i})} \right) \left( 4 \Vert u \Vert_{a, \Omega_i} +  2 \Vert u \Vert_{L^{2}(\Omega_{i})}\right) \Vert P_i u \Vert_{L^2(\Omega_i)}.
\end{align}

Now, using Lemma~\ref{lem:2-4-0} again, together with the Cauchy--Schwarz inequality, we obtain
\begin{align}
\sum_{i=1}^{N}\big[\big(\mathbf{b}\cdot\nabla &(R_{i}^{\top}T_{i}u-u), R^{\top}_{i}P_{i}u\big) + \left(c^-(R_{i}^{\top}T_{i}u-u), R^{\top}_{i}P_{i}u\right)\big] \nonumber \\
&\leq 4H \, C_1(\mathbf{b}, c^-) \, \left(\sum_{i=1}^{N}\left(\Vert u \Vert_{L^{2}(\Omega_{i})} + \Vert u \Vert_{a,\Omega_{i}} \right)^{2}\right)^{1/2} \left(\sum_{i=1}^{N}\Vert P_{i}u \Vert^{2}_{a,\Omega_{i}} \right)^{1/2} \nonumber\\
&\leq 8H \, C_1(\mathbf{b}, c^-) \, \left(\sum_{i=1}^{N} \Vert u \Vert_{a,\Omega_{i}} ^{2}\right)^{1/2} \left(\sum_{i=1}^{N}\Vert P_{i}u \Vert^{2}_{a,\Omega_{i}} \right)^{1/2} \nonumber\\
&\leq 8H k_{0}^{1/2} \, C_1(\mathbf{b}, c^-) \, \Vert u\Vert_{a} \left(\sum_{i=1}^{N}\Vert P_{i}u \Vert^{2}_{a,\Omega_{i}} \right)^{1/2}.\label{eq:4-22}
\end{align}
Hence, using \eqref{eq:2-29}, we obtain
\begin{align}
\label{eq:4-24}
\begin{split}
\beta^{2}_{0}(k_0\Theta^{1/2})&\sum_{i=1}^{N}\big[\big(\mathbf{b}\cdot\nabla (R_{i}^{\top}T_{i}u-u), \,R^{\top}_{i}P_{i}u\big) + \big(c^-(R_{i}^{\top}T_{i}u-u), R^{\top}_{i}P_{i}u\big)\big] \\
&\leq \frac{t}{\sqrt{2}(k_{0}+1)^{1/2}} \Vert u\Vert_{a} \left(\sum_{i=1}^{N}\Vert P_{i}u \Vert^{2}_{a,\Omega_{i}} \right)^{1/2}.
\end{split}
\end{align}
It then follows from \eqref{eq:4-20}, \eqref{eq:4-24}, that
\begin{align}
\beta^{2}_{0}(k_0\Theta^{1/2})&\sum_{i=0}^{N}\big[\big(\mathbf{b}\cdot\nabla(R_{i}^{\top}T_{i}u-u), \,R^{\top}_{i}P_{i}u\big) + \big(c^-(R_{i}^{\top}T_{i}u-u), R^{\top}_{i}P_{i}u\big)\big] \nonumber \\
&\leq \frac{\max\{s,t\}}{\sqrt{2}(k_{0}+1)^{1/2}} \left(\Vert u \Vert_{a} \Vert P_{0}u \Vert_{a}+ \Vert u\Vert_{a} \left(\sum_{i=1}^{N}\Vert P_{i}u \Vert^{2}_{a,\Omega_{i}} \right)^{\frac{1}{2}}\right) \nonumber \\
&\leq \frac{\max\{s,t\}}{(k_{0}+1)^{1/2}} \Vert u\Vert_{a} \left(\sum_{i=0}^{N}\Vert P_{i}u \Vert^{2}_{a,\Omega_{i}}  \right)^{\frac{1}{2}} \leq \max\{s,t\} \Vert u \Vert_{a}^{2}.
\label{eq:4-25}
\end{align}
Now, inserting \eqref{eq:4-25} into \eqref{eq:4-17}, we see that
\begin{equation}\label{eq:4-26}
a(u,u) \leq \beta_{0}^{2}(k_0\Theta^{1/2})a(Tu, u) + \max\{s,t\} \Vert u \Vert_{a}^{2},
\end{equation}
which implies that
\begin{equation}\label{eq:4-27}
\beta_{0}^{-2}(k_0\Theta^{1/2})\left(1-\max\{s,t\} \right)a(u,u) \leq a(Tu, u),
\end{equation}
thus proving \eqref{eq:2-26}.

Now, to prove \eqref{eq:2-27}, we observe that
\begin{equation}\label{eq:4-28}
a(Tu, Tu) = \left\Vert \sum_{i=0}^{N}R_{i}^{\top}T_{i}u \right\Vert_{a}^{2}\leq 2\Vert T_{0}u \Vert_{a}^{2} + 2\left\Vert \sum_{i=1}^{N}R_{i}^{\top}T_{i}u\right\Vert_{a}^{2},
\end{equation}
Further, by \eqref{eq:4-11}, we have
\begin{equation}\label{eq:4-29}
\Vert T_{0}u \Vert_{a}^{2} \leq 6\Vert u\Vert_{a}^{2},
\end{equation}
and in addition, by Lemma \ref{lem:4-0}, we obtain
\begin{align}
\left \Vert \sum_{i=1}^{N}R_{i}^{\top}T_{i}u\right\Vert_{a}^{2} \leq k_{0} \sum_{i=1}^{N}\left\Vert T_{i}u\right\Vert_{a,\Omega_{i}}^{2} &\leq 2 k_{0}\sum_{i=1}^{N}\left(6\Vert u \Vert^{2}_{a,\Omega_{i}} + 2\Vert u \Vert^{2}_{L^{2}(\Omega_{i})}\right) \nonumber \\
&\leq 2 k_{0}^{2}\left(6\Vert u \Vert^{2}_{a} + 2\Vert u \Vert^{2}_{L^{2}(\Omega)}\right) \leq 16 k_{0}^{2}\Vert u \Vert^{2}_{a}.
\label{eq:4-30}
\end{align}
Combining \eqref{eq:4-28}--\eqref{eq:4-30}, we finally determine that
\begin{equation}\label{eq:4-31}
a(Tu, Tu) \leq \left(12 + 32 k_{0}^{2}\right) \Vert u\Vert_{a}^{2},
\end{equation}
thus completing the proof of Theorem~\ref{thm:2-1}.
\end{proof}

Using Theorem~\ref{thm:2-1} and the  error estimates for GMRES in \cite{elman}, we obtain the following key result as a corollary.
\begin{corollary}
Under the assumptions of Theorem~\ref{thm:2-1}, if GMRES in the $\langle \cdot,\cdot\rangle_\mathbf{A}$-inner product is applied to solve the preconditioned system \eqref{eq:2-23} then, after $m$ steps, the norm of the residual is bounded as
\begin{equation}\label{eq:2-31}
\Vert \mathbf{r}^{m} \Vert^{2}_{\mathbf{A}} \leq \left(1-\frac{c_{1}^{2}}{c_{2}^{2}}\right)^{m}\Vert \mathbf{r}^{0} \Vert^{2}_{\mathbf{A}},
\end{equation}
where $\mathbf{r}^{0}$ is the initial residual.
\end{corollary}

\begin{proof}
The estimates \eqref{eq:2-26} and \eqref{eq:2-27}, together with \eqref{eq:equivalence} imply
\begin{align*}
\left\vert \langle\mathbf{M}_{AS,2}^{-1} \mathbf{B} \mathbf{u}, \mathbf{u} \rangle_{\mathbf{A}}\right\vert \geq c_1(H_0,\Lambda_0) \, \Vert \mathbf{u} \Vert _{\mathbf{A}}^2 \quad \text{and} \quad \Vert \mathbf{M}_{AS,2}^{-1} \mathbf{B} \Vert_{\mathbf{A}}^2 \leq c_2.
\end{align*}
The result follows directly from the Elman theory \cite{elman}.
\end{proof}
Restricting to the special case $\mathbf{b} = \mathbf{0}$ and $c = 0$, we obtain essentially the same result as in \cite{spillane:2014}, which is interesting because the latter makes use of `self-adjoint' technology, working with eigenvalue and condition number estimates and does not estimate the field of values as in \cite{elman}.

The fact that domain decomposition theory usually provides theoretical estimates in the energy inner product, while GMRES is normally applied with respect to the Euclidean inner product makes the estimate in the previous corollary slightly impractical. However this can be converted to a statement about standard GMRES using the following  norm equivalence argument (see also  \cite[Corollary 5.8]{ivan}).
\begin{lemma} \label{lem:norm_GMRES}
Suppose we are solving the linear system
\begin{align} \label{eq:Csystem}
\mathbf{C} \mathbf{x} = \mathbf{c}
\end{align}
on $\mathbb{R}^n$ and $\langle \cdot , \cdot \rangle_1$ and
$\langle \cdot , \cdot \rangle_2$  are two inner products on $\mathbb{R}^n$ with
associated norms $\Vert \cdot \Vert_1$ and $\Vert \cdot
\Vert_2$. The corresponding equivalence constants are denoted by $\underline{\gamma}, \overline{\gamma}$, so that
\begin{align} \label{eq:normeq}
\underline{\gamma} \Vert \mathbf{x} \Vert_1 \leq \Vert \mathbf{x}
\Vert_2\leq \overline{\gamma} \Vert \mathbf{x} \Vert_1, \quad
\text{for all} \quad \mathbf{x} \in \mathbb{R}^n.
\end{align}
Given a fixed initial guess $\mathbf{x}^0 \in \mathbb{R}^n$ with
residual $\mathbf{r}^0$, consider the  sequences of
iterates $\mathbf{x}_1^m$ and $\mathbf{x}_2^m$, $m = 0,1,2,\ldots$,
with residuals $\mathbf{r}_1^m$ and $\mathbf{r}_2^m$ that
minimise the residual of \eqref{eq:Csystem} over the Krylov subspace spanned by
$C$ in the norms $\Vert \cdot \Vert_1$ and
$\Vert \cdot \Vert_2$ respectively.
Suppose also that the sequence $\mathbf{r}_1^m$ enjoys the relative residual estimate
\begin{align*}
\Vert \mathbf{r}_1^m\Vert_1 \leq \sigma^m  \Vert \mathbf{r}^0\Vert_1, \quad \text{for some} \quad \sigma < 1.
\end{align*}
Then we have
\begin{equation}\label{eq:rmp}
\Vert \mathbf{r}_2^{m + \Delta m}\Vert_2 \leq \sigma^m \Vert
\mathbf{r}^0\Vert_2, \quad \text{for all} \quad \Delta m \geq
\frac{\log(\overline{\gamma}/\underline{\gamma})}{\log(\sigma^{-1})}.
\end{equation}
\end{lemma}
\begin{proof}
Since $\mathbf{x}_2$ minimises the residual in the norm $\Vert
\cdot \Vert_2$ and $\mathbf{r}_1^0 = \mathbf{r}_2^0 = \mathbf{c} - C
\mathbf{x}^0$, it follows from \eqref{eq:normeq} that
\begin{align*}
\frac{\Vert \mathbf{r}^{m + \Delta m}_2\Vert_2}{\Vert \mathbf{r}_2^{0}\Vert_2} \le \frac{\Vert \mathbf{r}_1^{m + \Delta m}\Vert_2}{\Vert \mathbf{r}_1^{0}\Vert_2}\le \frac{\overline\gamma}{\underline\gamma} \frac{\Vert \mathbf{r}^{m + \Delta m}_1\Vert_1}{\Vert \mathbf{r}^{0}_1\Vert_1} \le \frac{\overline\gamma}{\underline\gamma} \sigma^{m + \Delta m}.
\end{align*}
Thus, we can deduce the bound in \eqref{eq:rmp} if $\Delta m \log(\sigma^{-1}) \ge \log(\overline{\gamma}/\underline\gamma)$.
\end{proof}
\begin{corollary} \label{cor:Euclidean}
Assuming the mesh sequence $\mathcal{T}_h$ is quasiuniform, then
with an additional number $\Delta m$ of iterations,
which grows at most proportionally to $\log(a_{\max})+ \log(h^{-1})$, GMRES applied in the Euclidean inner product
(i.e., standard GMRES) for the preconditioned system \eqref{eq:2-23} satisfies the bound
\begin{equation*}
\Vert \mathbf{r}^{m+\Delta m} \Vert^{2} \leq \left(1-\frac{c_{1}^{2}}{c_{2}^{2}}\right)^{m}\Vert \mathbf{r}^{0} \Vert^{2}.
\end{equation*}
\end{corollary}
\begin{proof}
This uses the previous lemma with $\Vert \cdot \Vert_1 = \Vert \cdot \Vert_A$
and $\Vert \cdot \Vert_{2} = \Vert \cdot \Vert$ (the Euclidean norm), together with inverse and norm equivalence estimates  which imply \eqref{eq:normeq} with
$\underline{\gamma} = \underline{C} h^{d/2} $ and
$\overline{\gamma} = \overline{C} \sqrt{a_{\max}} h^{d/2-1} $,
where $\underline{C}$ and $\overline{C}$ are constants that are independent of all parameters.
\end{proof}

%% file: numerical_results_revision.tex
\section{Numerical experiments}
\label{sec:numerical}

In this section we provide numerical results primarily for the two model problems: Firstly, an indefinite elliptic problem (where $\kappa > 0$)
\begin{subequations}
\label{eq:helmholtz}
\begin{align}
-\mathrm{div}( A \nabla u) - \kappa u &= f & & \text{in } \Omega, \\
u &= 0 & & \text{on } \partial\Omega,
\end{align}
where the splitting in \eqref{eq:splitc} used is given by
\begin{align}
\label{eq:helmholtz-splitting}
c^+ =0 \quad \text{and} \quad  c^- = - \kappa.
\end{align}
\end{subequations}
Secondly, a convection--diffusion problem
\begin{subequations}
\label{eq:convdiff}
\begin{align}
-\mathrm{div}( A \nabla u) + \mathbf{b}\cdot\nabla u &= f & & \text{in } \Omega, \\
u &= 0 & & \text{on } \partial\Omega,
\end{align}
where the splitting in \eqref{eq:splitc} used is given by
\begin{align}
\label{eq:convdiff-splitting}
c^+ = 0 = c^-.
\end{align}
\end{subequations}
Unless stated otherwise, in all our computations $\Omega = (0,1)^2$.

We first detail the common numerical setup. Our computations are performed using FreeFem \cite{hecht:2012} (\url{http://freefem.org/}), in particular within the \texttt{ffddm} framework. We perform experiments in 2D and discretise on uniform square grids with mesh spacing $h$, triangulating along one set of diagonals to form P1 elements. As the right-hand side $f$ we impose a point source at the centre of the domain $\Omega$.

To precondition the large sparse linear systems that arise from
discretisation, unless otherwise stated, we utilise a two-level additive Schwarz preconditioner
where the GenEO coarse space is incorporated additively, as in
\eqref{eq:2-24}. This is used within a right-preconditioned GMRES
iteration, which is terminated once a relative reduction of the
residual of at least $10^{-6}$ has been achieved, or otherwise after a
maximum of 1000 iterations (denoted in the tables below by `$1000+$').
For the local subdomains
we use a uniform decomposition into $N$ square subdomains which,
unless stated otherwise, are then extended by adding only the
fine-mesh elements which touch them (yielding a `minimal overlap'
scenario). Within the GenEO coarse space we take all eigenvectors
corresponding to eigenvalues satisfying the eigenvalue threshold
$\lambda < \lambda_\text{max}$. Unless otherwise stated, we use the
threshold parameter $\lambda_\text{max} = 0.5$ which, in light of
Definition \ref{def:2-2}, ensures that $\Theta \le 2$ in our
numerical experiments. Note that this means we do not make $\Theta$
particularly small in practice, although it is necessary in
theory. To solve the local generalised eigenvalue problems we use
ARPACK \cite{arpack:1998}, as interfaced through FreeFem.

\subsection{The case of homogeneous diffusion coefficient}
\label{subsec:numerics-homogeneous}

To illustrate some of the basic properties of the GenEO method, we first consider some simple homogeneous
test cases (i.e., with $A = I$), exploring how  indefiniteness/non-self-adjointness  affects performance.
In Table~\ref{Table:Indefinite_homogeneous_nglob600} we display results for problem \eqref{eq:helmholtz},
varying the strength of the indefiniteness through increasing $\kappa$
\nb{while fixing $h$ (left) or varying
  $h \propto 1/\sqrt{\kappa}$ (right).
  As $\kappa$ grows,  the solution  will be rich in the Laplace  Dirichlet eigenmodes corresponding to eigenvalues near  ${\kappa}$, and so this scaling 
  is analogous to choosing    a fixed number of points
per wavelength in a Helmholtz problem.  (In fact, here, approximately 100 points per wavelength).} We see that the
GenEO coarse space, built only from the positive
(Laplacian) part of the problem (recall \eqref{eq:helmholtz-splitting}), works surprisingly well and
provides a scalable solver for $\kappa \leq 1000$.
However, once $\kappa$ is too large the performance degrades and for
$\kappa = 10000$ the method lacks robustness and scalability, the
underlying problem now being highly indefinite.
\nb{In comparing the left and right parts of Table 1, we see that choosing   $h \propto (\sqrt{\kappa})^{-1}$ only mildly affects iteration counts  for the range of $\kappa$ in  which the method works well. 
}

\begin{table}[t]
	\centering
	\small
	\tabulinesep=1.2mm
	\begin{tabu}{cc|ccccc}
		& & \multicolumn{5}{c}{$N$} \\
		$\kappa$ & $h$ & 4  & 16  & 36  & 64  & 100 \\
		\hline
		1 & $1/600$ & 16  & 17 & 17 & 18 & 18 \\
		10 & $1/600$ & 17  & 18 & 18 & 18 & 18 \\
		100 & $1/600$ & 24  & 27 & 26 & 23 & 23 \\
		1000 & $1/600$ & 40  & 98 & 102 & 113 & 89 \\
		10000 & $1/600$ & 144  & 431 & 660 & $\!\!1000+\!\!$ & $\!\!1000+\!\!$
	\end{tabu}
	\hfill
	\begin{tabu}{cc|ccccc}
	& & \multicolumn{5}{c}{$N$} \\
	$\kappa$ & $h$ & 4  & 16  & 36  & 64  & 100 \\
	\hline
	1 & $1/16$ & 14 & 15 & $-$ & $-$ & $-$ \\
	10 & $1/51$ & 16 & 18 & 19 & 22 & 24 \\
	100 & $1/160$ & 25 & 25 & 23 & 25 & 24 \\
	1000 & $1/506$ & 59 & 118 & 147 & 130 & 106 \\
	10000 & $1/1600$ & 161 & 446 & 821 & $\!\!1000+\!\!$ & $\!\!1000+\!\!$
	\end{tabu}
	\caption{Iteration counts for the homogeneous, elliptic, indefinite problem \eqref{eq:helmholtz}, varying the parameter $\kappa$ and the number of subdomains $N$ for fixed $h = 1/600$ (left) and $h \propto 1/\sqrt{\kappa}$ (right).}
	\label{Table:Indefinite_homogeneous_nglob600}
\end{table}

Turning now to non-self-adjointness, Table~\ref{Table:Convection_zero_div_homogeneous_nglob600} (left) gives results for \eqref{eq:convdiff} with $A = I$ and
\begin{align}
\label{conv:zero-div}
\mathbf{b} = b \, (1+\sin(2\pi(2y-x))) (2,1)^{T},
\end{align}
\nb{for a fixed $h$.}
The non-self-adjointness can be increased by varying $b$.
In this case, $\nabla \cdot \mathbf{b} = 0$ and so in
\eqref{eq:constants} $C_2(\mathbf{b},
c^-) = C_1(\mathbf{b}, c^-)$, thus improving
the estimate in Theorem \ref{thm:2-1}. We observe that, also in this case,
the GenEO coarse space works well and provides
a scalable method for $b$ up to at least $1000$. For $b = 10000$
the convection term appears too strong and the iteration counts
deteriorate.

\begin{table}[t]
	\centering
	\tabulinesep=1.2mm
	\begin{tabu}{c|ccccc}
		& \multicolumn{5}{c}{$N$} \\
		$b$ & 4 & 16 & 36 & 64 & 100 \\
		\hline
		1 & 18 & 18 & 18 & 18 & 18 \\
		10 & 28 & 26 & 23 & 22 & 21 \\
		100 & 35 & 34 & 30 & 28 & 27 \\
		1000 & 57 & 59 & 63 & 62 & 59 \\
		10000 & 172 & 276 & 355 & 472 & 512
	\end{tabu} \qquad\qquad
	\begin{tabu}{c|ccccc}
		& \multicolumn{5}{c}{$N$} \\
		$b$ & 4 & 16 & 36 & 64 & 100 \\
		\hline
		1 & 18 & 18 & 18 & 18 & 18 \\
		10 & 28 & 26 & 23 & 22 & 21 \\
		100 & 39 & 43 & 35 & 29 & 25 \\
		1000 & 72 & 107 & 71 & 74 & 63 \\
		10000 & 201 & 450 & 708 & 835 & 837
	\end{tabu}
\caption{Iteration counts for the convection--diffusion
          problem \eqref{eq:convdiff}, with homogeneous diffusion $A = I$ and
          convection coefficient \eqref{conv:zero-div} with zero
          divergence (left) and \eqref{conv:non-zero-div} with
          nonzero divergence (right), varying the constant $b$ and
        the number of subdomains $N$ for fixed $h = 1/600$.}
	\label{Table:Convection_zero_div_homogeneous_nglob600}
\end{table}

We also consider problem \eqref{eq:convdiff} with a convection
field with non-vanishing divergence, namely
\begin{align}
\label{conv:non-zero-div}
\mathbf{b} = b \, (1+\sin(2\pi(2x+y))) (2,1)^{T}.
\end{align}
These results are given in Table~\ref{Table:Convection_zero_div_homogeneous_nglob600} (right) and
are similar to those in the left table,
displaying robustness except in the most convection-dominated case.
However, for higher $b$ the iteration counts become larger compared to the zero-divergence case,
illustrating that the strength of  $\nabla \cdot \mathbf{b}$ does play a role, as
predicted by the theory through the constant $C_2$ (given in
\eqref{eq:constants}). We consider the impact of the convection field
further in Section~\ref{subsec:numerics-convfield}.

\nb{It is well-known that  specialist discretisations of `upwind' type, are required in order to obtain accurate solutions for practical meshes  when the convection becomes large. 
  However such approaches are typically not of the standard Galerkin form \eqref{eq:2-7}, and so an
  extension of the theory given here would be needed to treat these.  Nevertheless
  since the  standard Galerkin method treated here is effective if $h$ is small enough relative to $b$, some sample results with   $h \propto b^{-1}$ are given in Table \ref{Table:Convection_zero_div_homogeneous_nglob_varying}.  Here,  for  illustration,  we compute only up to   $b=100$.
  Comparing the results from this refinement strategy to the fixed mesh results in Table \ref{Table:Convection_zero_div_homogeneous_nglob600}, we see that the iteration counts have grown   only mildly, and the GenEO approach remains very useful for moderate $b$.}

\nb{We explore the dependence on mesh refinement further in Section~\ref{subsec:numerics-meshrefinement}; otherwise, in light of these results, we will typically consider $h$ as fixed.}

\begin{table}[t]
	\centering
	\tabulinesep=1.2mm
	\begin{tabu}{cc|ccccc}
		& & \multicolumn{5}{c}{$N$} \\
		$b$ & $h$ & 4 & 16 & 36 & 64 & 100 \\
		\hline
		1 & $1/16$ & 14 & 15 & $-$ & $-$ & $-$ \\
		10 & $1/160$ & 23 & 21 & 21 & 20 & 20 \\
		100 & $1/1600$ & 58 & 54 & 48 & 46 & 42
	\end{tabu} \qquad\qquad
	\begin{tabu}{cc|ccccc}
		& & \multicolumn{5}{c}{$N$} \\
		$b$ & $h$ & 4 & 16 & 36 & 64 & 100 \\
		\hline
		1 & $1/16$ & 14 & 15 & $-$ & $-$ & $-$ \\
		10 & $1/160$ & 23 & 22 & 20 & 20 & 19 \\
		100 & $1/1600$ & 59 & 69 & 56 & 46 & 40
	\end{tabu}
	\caption{Iteration counts for the convection--diffusion
		problem \eqref{eq:convdiff}, with homogeneous diffusion $A = I$ and
		convection coefficient \eqref{conv:zero-div} with zero
		divergence (left) and \eqref{conv:non-zero-div} with
		nonzero divergence (right), varying the constant $b$ and
		the number of subdomains $N$ for $h \propto 1/b$.}
	\label{Table:Convection_zero_div_homogeneous_nglob_varying}
\end{table}


\subsection{The case of a heterogeneous diffusion coefficient}
\label{subsec:numerics-heterogeneous}

\begin{figure}[t]
	\centering
	\begin{tikzpicture}[scale=0.6]
		\foreach \x in {0,2,4,6,8}
		\foreach \y in {0,2,4,6,8}
		\fill[gray,opacity=((\y+1)*0.1)] (\x,\y) -- (\x+1,\y) -- (\x+1,\y+1) -- (\x,\y+1) -- cycle;
		\fill[gray,opacity=1] (9*0.16,9*0.63) -- (9*0.16,9*0.68) -- (9*0.72,9*0.47) -- (9*0.72,9*0.42) -- cycle;
		\fill[white,opacity=1] (9*0.22,9*0.03) -- (9*0.22,9*0.11) -- (9*0.94,9*0.23) -- (9*0.94,9*0.15) -- cycle;
		\fill[gray,opacity=0.5] (9*0.22,9*0.03) -- (9*0.22,9*0.11) -- (9*0.94,9*0.23) -- (9*0.94,9*0.15) -- cycle;
		\draw (0,0) -- (9,0) -- (9,9) -- (0,9) -- cycle;
	\end{tikzpicture}
	\caption{The heterogeneous function $a(\mathbf{x})$ within the inclusions and channels test problem. The shading gives the value of $a(\mathbf{x})$ with the darkest shade being $a(\mathbf{x}) = a_\text{max}$, where $a_\text{max}>1$ is a parameter, and the background white taking the value $a(\mathbf{x}) = 1$.}
	\label{Fig:InclusionsAndChannels}
\end{figure}
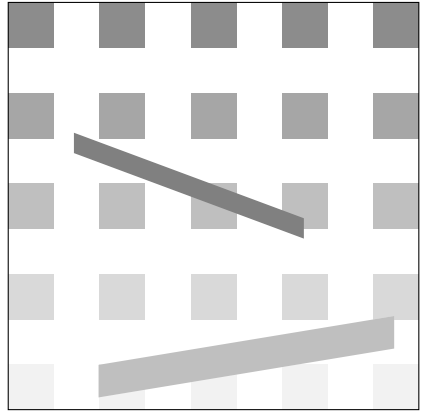
We now consider the same experiments as in the previous subsection but with the diffusion coefficient $A$ being heterogeneous. In particular, we will consider an ``inclusions and channels'' problem, akin to that in \cite{nataf:2011}. Here, the heterogeneous coefficient is given by
\begin{align}
A(\mathbf{x}) = a(\mathbf{x}) I, \quad \text{where }
a(\mathbf{x}) \text{ is given in Figure~\ref{Fig:InclusionsAndChannels}}. \label{eq:hetcoeff}
\end{align}
The function $a(\mathbf{x})$ represents a number of inclusions, whose contrast
with the background value of $a_\text{min} = 1$ varies depending
on the $y$-coordinate, along with two channels which run across the
inclusions and link through a larger part of the domain, one with
$a(\mathbf{x}) = a_\text{max}$ and one with $a(\mathbf{x}) =
\frac{1}{2} a_\text{max}$. This represents a challenging heterogeneous
problem for iterative solvers even in the self-adjoint and
positive-definite case (i.e., $\mathbf{b} = 0$ or $\kappa = 0$).
The strength of the heterogeneity is determined by the contrast parameter $a_\text{max}$.

Results for problem \eqref{eq:helmholtz} are given in Table~\ref{Table:Heterogeneous_nglob600} (left)
while results for problem \eqref{eq:convdiff} with zero-divergence
$\mathbf{b}$ are in Table~\ref{Table:Heterogeneous_nglob600}
(right). Results for the case when $\nabla \cdot \mathbf{b} \not=0$ are
similar to those in the zero-divergence case,  and thus
omitted. We see that in all cases where the GenEO coarse space led
to a robust method for the homogeneous problem, it does also for the
heterogeneous problem and, moreover, is robust to the strength of
the heterogeneity (here determined by $a_\text{max}$). In the highly
indefinite or highly non-self-adjoint cases, where the method
struggles already for homogeneous problems, the lower iteration
counts for some of the heterogeneous problems may stem from the fact
that in some parts of the domain the heterogeneity in fact reduces
the strength of the indefiniteness or of the convection term
locally.  For instance, consider the part $\Omega'$ of the domain
where $a$ is constant and has the largest value $a_\text{max}$
(indicated by the darkest shade in Figure
\ref{Fig:InclusionsAndChannels}). Then, restricting
\eqref{eq:helmholtz} to $\Omega'$, we see  that effectively \eqref{eq:helmholtz} reduces to
\begin{align*}
-\Delta u - \left(a_\text{max}^{-1}\kappa\right) u = a_\text{max}^{-1} f,
\end{align*}
hence reducing the effective indefiniteness in this region when $a_\text{max} > 1$ is large.
\begin{table}[t]
	\centering
	\tabulinesep=1.2mm
	\begin{tabu}{cc|ccccc}
		\multicolumn{2}{c|}{} & \multicolumn{5}{c}{$N$} \\
		$\kappa$ & $a_\text{max}$ & 4 & 16 & 36 & 64 & 100 \\
		\hline
		1 & 5  & 18 & 18 & 18 & 18 & 18   \\
		1 & 10 & 18 & 18 & 18 & 18 & 19   \\
		1 & 50 & 18 & 18 & 19 & 19 & 19   \\
		\hline
		10 & 5  & 18 & 18 & 18 & 19 & 19   \\
		10 & 10 & 18 & 18 & 19 & 19 & 19   \\
		10 & 50 & 18 & 18 & 20 & 19 & 19   \\
		\hline
		100 & 5  & 27 & 30 & 26 & 23 & 22   \\
		100 & 10 & 26 & 26 & 24 & 23 & 22   \\
		100 & 50 & 23 & 22 & 22 & 20 & 20   \\
		\hline
		1000 & 5  & 70 & 140 & 177 & 167 & 138   \\
		1000 & 10 & 58 & 118 & 146 & 131 & 108   \\
		1000 & 50 & 48 & 94  & 114 & 106 & 99    \\
		\hline
		10000 & 5  & 193 & 524 & 906 & $\!\!1000+\!\!$ & $\!\!1000+\!\!$  \\
		10000 & 10 & 162 & 457 & 815 & $\!\!1000+\!\!$& $\!\!1000+\!\!$ \\
		10000 & 50 & 114 & 318 & 704 & 888      & $\!\!1000+\!\!$
	\end{tabu}\quad
	\begin{tabu}{cc|ccccc}
		\multicolumn{2}{c|}{} & \multicolumn{5}{c}{$N$} \\
		 $b$ & $a_\text{max}$ & 4 & 16 & 36 & 64 & 100 \\
		\hline
		 1 & 5  & 18 & 18 & 18 & 18 & 18 \\
		 1 & 10 & 18 & 18 & 19 & 19 & 19 \\
		 1 & 50 & 18 & 18 & 19 & 19 & 19 \\
		\hline
		10 & 5  & 25 & 25 & 23 & 22 & 21 \\
		10 & 10 & 24 & 24 & 23 & 21 & 20 \\
		10 & 50 & 23 & 23 & 23 & 21 & 21 \\
		\hline
		100 & 5  & 39 & 36 & 31 & 28 & 27 \\
		100 & 10 & 38 & 36 & 31 & 27 & 26 \\
		100 & 50 & 32 & 34 & 29 & 27 & 26 \\
		\hline
		1000 & 5  & 51 & 59 & 57 & 57 & 54 \\
		1000 & 10 & 49 & 53 & 54 & 56 & 53 \\
		1000 & 50 & 53 & 57 & 55 & 54 & 51 \\
		\hline
		10000 & 5  & 142 & 232 & 298 & 370 & 408 \\
		10000 & 10 & 121 & 180 & 261 & 320 & 328 \\
		10000 & 50 & 87  & 148 & 224 & 261 & 278
	\end{tabu}
	\caption{Iteration counts for the heterogeneous, elliptic,
          indefinite problem (left) and for the convection--diffusion
          problem with zero-divergence convection coefficient
          \eqref{conv:zero-div} (right), varying the parameter
          $\kappa$ (left) and the constant $b$ (right), as well as
          $a_\text{max} > 1$ and the number of subdomains $N$ for fixed $h = 1/600$.}
	\label{Table:Heterogeneous_nglob600}
\end{table}

In Table~\ref{Table:Coarse_space_size_nglob600} we display the size of
the GenEO coarse space employed, along with comparison figures from
the homogeneous case (labelled $a_\text{max} = 1$). We note that, at
most, only a small number of additional eigenvectors are required in
the coarse space in order to treat the heterogeneity for any
$a_\text{max}$ used, thus showing robustness to contrasts in the
diffusion coefficient. We also see that the size of the coarse space
grows with $N$, however, the average number of eigenvectors per
subdomain decreases with $N$ so that we need to solve only small
eigenvalue problems for a small number of eigenvectors.
\nb{The size of the coarse space is also very small compared to the
global problem size, here ranging from $0.06\%$ to $0.5\%$ as
$N$ increases from 4 to 100.}

In conclusion, the GenEO method is well-equipped to deal with heterogeneities in the coefficient $A$, allowing us to treat the heterogeneous problem similarly to the homogeneous one.

\begin{table}[ht]
	\centering
	\tabulinesep=1.2mm
	\begin{tabu}{c|ccccc}
		& \multicolumn{5}{c}{$N$} \\
		$a_\text{max}$ & 4 & 16 & 36 & 64 & 100 \\
		\hline
		1  & 212 (53.0) & 624 (39.0) & 1060 (29.4) & 1480 (23.1) & 1800 (18.0) \\
		10 & 209 (52.3) & 613 (38.3) & 1046 (29.1) & 1447 (22.6) & 1828 (18.3) \\
		50 & 202 (50.5) & 612 (38.3) & 1055 (29.3) & 1442 (22.5) & 1845 (18.5)
	\end{tabu}
	\caption{The size of the coarse space for the homogeneous
          problem (denoted $a_\text{max} = 1$) and the heterogeneous
          inclusions and channels problem, varying $a_\text{max}$ and
          the number of subdomains $N$ for fixed $h = 1/600$. The average number of eigenvectors per subdomain, namely the size of the coarse space divided by $N$, is given in parenthesis. \nb{The total number of degrees of freedom for this problem is $361,\!201$, yielding coarse space sizes ranging from $0.06\%$ to $0.5\%$ of the global problem size as $N$ increases from 4 to 100.}}
	\label{Table:Coarse_space_size_nglob600}
\end{table}

\subsection{Dependence on mesh refinement}
\label{subsec:numerics-meshrefinement}

In this subsection we explore dependence upon the mesh spacing $h$. We consider problem  \eqref{eq:helmholtz}
with $A$ given as in \eqref{eq:hetcoeff} and we set  $\kappa = 1000$
and $a_\text{max} = 50$. Table~\ref{Table:Indefinite_heterogeneous_kappa1000_amax50} (left)
shows how the iteration count and the size of the coarse space vary with $h$. We
note that the iteration count remains stable as
$h$ decreases and the number $N$ of subdomains grows. On
the other hand, the size of the coarse space appears to
grow at a rate slightly slower than $\mathcal{O}(h^{-1})$ as $h$ is
decreased. This is in full agreement with similar results in
\cite[Table 1]{peter} for positive-definite problems in the case of minimal
overlap $\delta \sim h$.

As in \cite{peter}, we repeat our experiment for the case of
generous overlap $\delta \sim H$. We fix the subdomain size $H$ and
the overlap $\delta$ independent of $h$, in particular to that given
by minimal overlap on the coarsest mesh size $h=\frac1{200}$, such that $\delta
=\frac{1}{100}$. The partition of unity is chosen to vary smoothly
over the entire extent of the overlap. (For details of the implementation
within FreeFem, see \cite[Section 4.1]{jolivet:2012}.) The iteration counts presented
in Table~\ref{Table:Indefinite_heterogeneous_kappa1000_amax50} (right)
for this setup remain stable and are essentially identical to those for minimal overlap
in the left table, but in addition, we now also observe that the size of the coarse space
remains bounded as $h\rightarrow 0$, which is again in agreement with the
results in \cite[Table 1]{peter}.

In all cases, our results show  that the dimension of the coarse space depends on $N$,
and grows approximately as $\mathcal{O}(N^{2/3})$; this means the
number of eigenvectors chosen per subdomain decreases with $N$ at
approximately the rate $\mathcal{O}(N^{-1/3})$. \nb{Further, Table~\ref{Table:Indefinite_heterogeneous_kappa1000_amax50} (below) provides the size of the coarse space as a percentage of the global problem size and we note that this is both small (at most $1.48\%$) and decreases with $h$.}

\begin{table}[t]
	\centering
	\tabulinesep=1.2mm
	\begin{tabu}{c|ccc|ccc}
		& \multicolumn{3}{c|}{$N$} & \multicolumn{3}{c}{$N$}\\
		$h$ & 16 & 36 & 100 & 16 & 36 & 100\\
		\hline
		$\frac{1}{200}$ & 87 (185) & 115 (318) & 90 (598) &
				87 (185) & 115 (318) & 90 (598) \\
		$\frac{1}{400}$ & 89 (419) & 107 (661) & 91 (1214) &
                88 (225) & 109 (356) & 94 (639) \\
		$\frac{1}{600}$ & 94 (612) & 114 (1055) & 99 (1845) &
                92 (227) & 116 (384) & 99 (672) \\
		$\frac{1}{800}$ & 95 (847) & 115 (1329) & 100 (2478) &
                93 (228) & 116 (379) & 100 (680) \\
        \hline
        $\frac{1}{200}$ & $0.46\%$ & $0.79\%$ & $1.48\%$ & $0.46\%$ & $0.79\%$ & $1.48\%$ \\
        $\frac{1}{400}$ & $0.26\%$ & $0.41\%$ & $0.75\%$ & $0.14\%$ & $0.22\%$ & $0.40\%$ \\
        $\frac{1}{600}$ & $0.17\%$ & $0.29\%$ & $0.51\%$ & $0.06\%$ & $0.11\%$ & $0.18\%$ \\
        $\frac{1}{800}$ & $0.13\%$ & $0.21\%$ & $0.39\%$ & $0.04\%$ & $0.06\%$ & $0.11\%$ \\
	\end{tabu}
        %
	\caption{\nb{Above:} Iteration count and size of the coarse space (in brackets) for
          minimal overlap $\delta \sim h$ (left) and for generous
          overlap $\delta \sim H$ (right) for problem \eqref{eq:helmholtz}
          with heterogeneous $A$ as in \eqref{eq:hetcoeff}, varying $h$ and
          $N$ for $\kappa = 1000$ and $a_\text{max} = 50$ fixed. \nb{Below: The size of the coarse space as a percentage of the global problem size.}}
	\label{Table:Indefinite_heterogeneous_kappa1000_amax50}
\end{table}

\subsection{The effect of the convection field}
\label{subsec:numerics-convfield}

For problem \eqref{eq:convdiff} the specific form of $\mathbf{b}$  can play an  important role.
Here we consider some further examples to illustrate this point and show how the divergence of $\mathbf{b}$ can dictate performance, as predicted by the theory.
We start with the  homogeneous case $A = I$,  and with a circulating convection field. Here the problem is posed on $(-1,1)\times(0,1)$ and the convection coefficient is given as
\begin{align}
\label{conv:circulating}
\mathbf{b} = b \, (2y(1-x^2),-2x(1-y^2))^{T},
\end{align}
illustrated in Figure~\ref{fig:conv_circ}. Results varying $b$, which
controls the strength of the convective term, are given in Table
\ref{Table:Convection_circulating_homogeneous_nglob600} (left). As before,
GenEO  behaves robustly  for  $b$ up to at least $1000$.

\begin{figure}[t]
	\centering
	\subfigure[Without oscillations: $\mathbf{b}$ given by \eqref{conv:circulating}.\label{fig:conv_circ}]{\includegraphics[width=0.49\textwidth,trim=1.8cm 0cm 1.8cm 0cm ,clip]{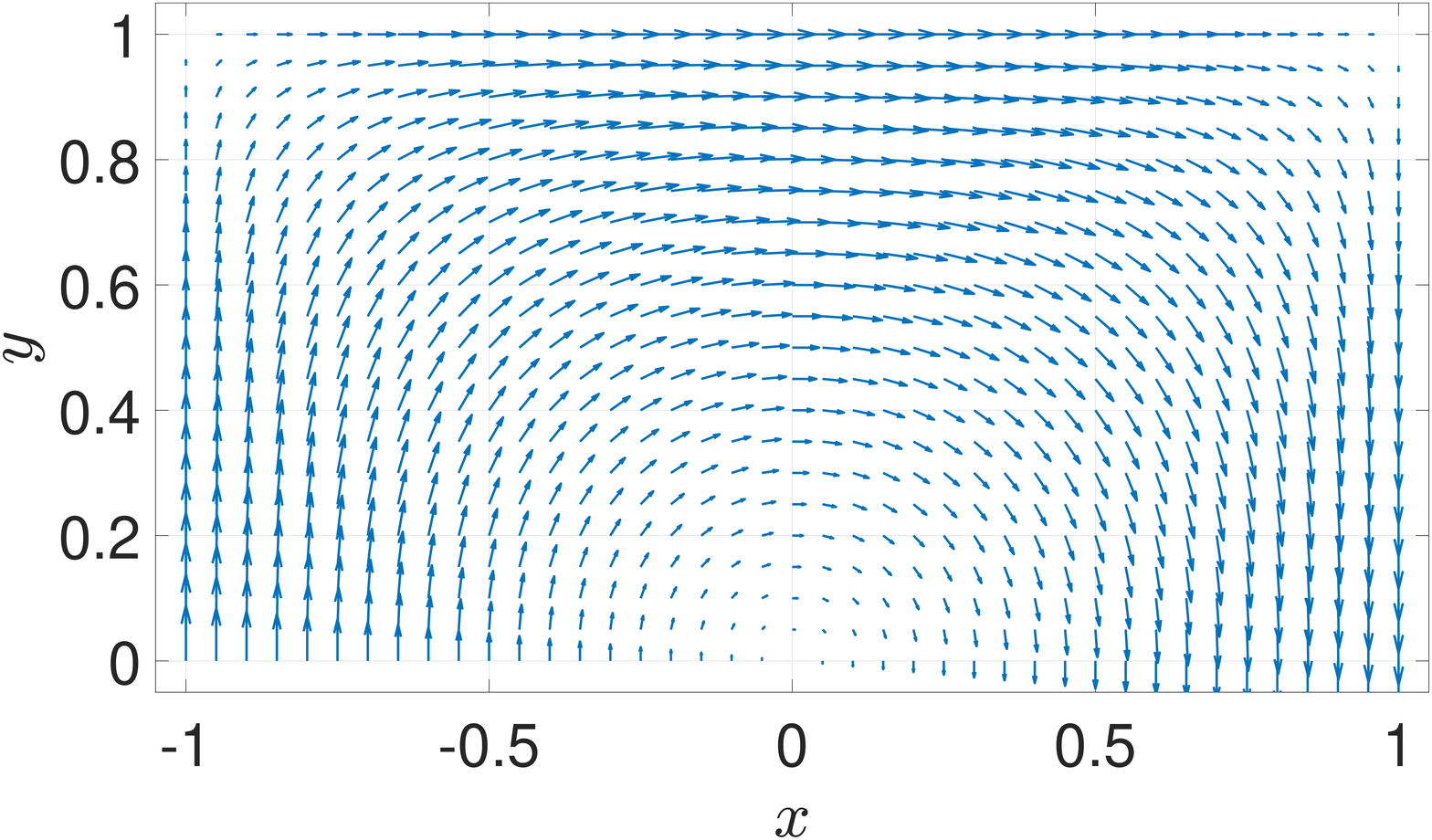}}
	\subfigure[With oscillations: $\mathbf{b}$ given by \eqref{conv:circulating-sine-radial} with $n=4$.\label{fig:conv_circ_osc}]{\includegraphics[width=0.49\textwidth,trim=1.8cm 0cm 1.8cm 0cm ,clip]{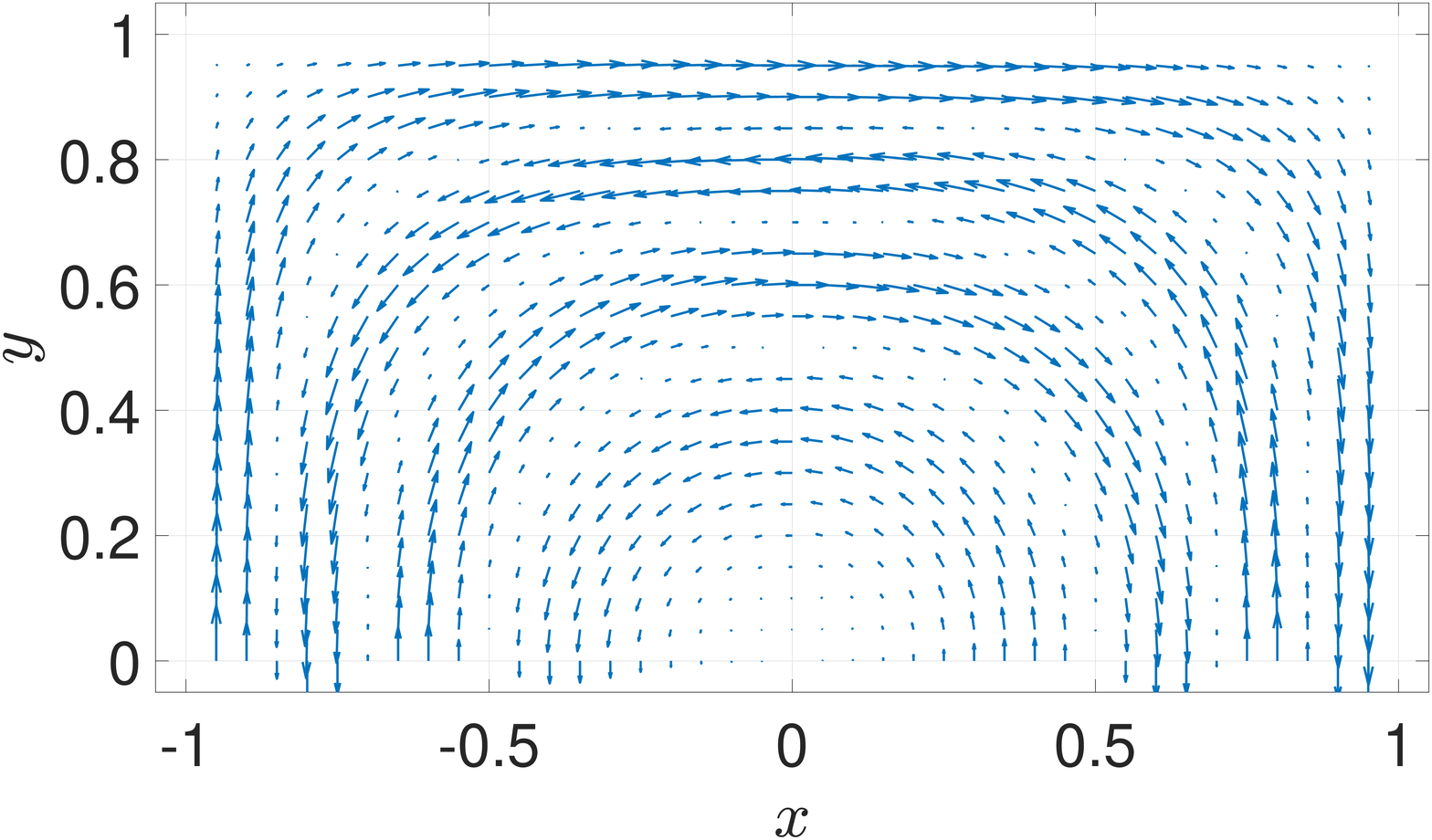}}
	\caption{Circulating convection fields with and without oscillation in the radial direction.}
	\label{fig:conv_fields_circ}
\end{figure}

\begin{table}[t]
	\centering
	\tabulinesep=1.2mm
	\begin{tabu}{c|ccccc}
		& \multicolumn{5}{c}{$N$} \\
		$b$ & 4 & 16 & 36 & 64 & 100 \\
		\hline
		1 & 18 & 18 & 18 & 18 & 18 \\
		10 & 24 & 24 & 22 & 21 & 20 \\
		100 & 43 & 44 & 41 & 36 & 33 \\
		1000 & 31 & 34 & 35 & 35 & 36 \\
		10000 & 143 & 231 & 282 & 314 & 339 \\
		\hline
		$-$ & 316 & 931 & 1584 & 2170 & 2796
	\end{tabu} \qquad\qquad
	\begin{tabu}{c|ccccc}
		& \multicolumn{5}{c}{$N$} \\
		$n$ & 4 & 16 & 36 & 64 & 100 \\
		\hline
		1 & 30 & 33 & 35 & 35 & 36 \\
		2 & 35 & 44 & 49 & 45 & 49 \\
		4 & 44 & 64 & 65 & 57 & 57 \\
		8 & 46 & 66 & 70 & 65 & 60
	\end{tabu}
	\caption{Iteration counts for problem \eqref{eq:convdiff} with
          $A= I$ and  with circulating convection coefficient for
            $h = 1/600$. Left:  Using \eqref{conv:circulating} and
            varying the parameter $b$ and the number of subdomains
            $N$; the coarse space size (independent of $b$) is
            tabulated in the final row. Right: Using
            \eqref{conv:circulating-sine-radial} with radial
            oscillations and $b=1000$, varying the parameter $n$ and the
            number of subdomains $N$.}
	\label{Table:Convection_circulating_homogeneous_nglob600}
\end{table}

As well as varying the strength of the convective term, we can
also add local variation. To this end,  we  add sinusoidal
variation of increasing frequency in the radial direction
(perpendicular to the characteristics
of the convection field). That is, we consider
\begin{align}
\label{conv:circulating-sine-radial}
\mathbf{b} = b \, \sin( n\pi(1-x^2)(1-y^2) ) (2y(1-x^2),-2x(1-y^2))^{T},
\end{align}
for a parameter $n$; see Figure~\ref{fig:conv_circ_osc}. Note that for
any choice of $n$ this convection coefficient is divergence-free
since $\nabla \cdot \mathbf{b} = 0$. Thus, here we
illustrate the effect of local variation without changing the
divergence. Results for $b=1000$ and varying $n$ are given in Table
\ref{Table:Convection_circulating_homogeneous_nglob600} (right),
where we see that iteration counts slightly increase with the
frequency parameter $n$ but that this remains mild and the coarse
space continues to provide a robust preconditioner as $N$ increases.

We now explore the role the local variation of the divergence of
$\mathbf{b}$ plays by considering a simple unidirectional convection
field with sinusoidal variation added in such a way as to vary
$\nabla \cdot \mathbf{b}$. In particular, we consider
\begin{align}
\label{conv:unidirectional-oscillations}
\mathbf{b} = b \, (1+\sin(m\pi(2x+y))) (1+\sin(2\pi(2y-x))) (2,1)^{T},
\end{align}
see Figure~\ref{fig:conv_field_unidir}, on the unit square. In this
case, we also choose the heterogeneous diffusion coefficient $A$ in
\eqref{eq:hetcoeff} to incorporate the full range of heterogeneity.

Here $\mathbf{b}$ includes oscillation both
across characteristics and along characteristics. The term
$(1+\sin(2\pi(2y-x)))$ introduces oscillations that do not
change the strength of the divergence. On the
other hand, the term $(1+\sin(m\pi(2x+y)))$ introduces oscillations
along characteristics and as we increase their frequency through the
parameter $m$ we increase the strength of the divergence locally. This can be
seen visually in Figure~\ref{fig:conv_field_unidir} where we have
$m=4$ oscillations along any unit vector parallel to $(2,1)^{T}$ and
$2$ oscillations along any unit vector perpendicular to
$(2,1)^{T}$. To see how this affects performance we fix $b = 1000$ and
$a_\text{max} = 50$ and display results for increasing $m$ in Table
\ref{Table:Convection_unidirectional_oscillating_heterogeneous_nglob600}. This
time we see that increasing the oscillations has a significant impact
on the iteration counts and the GenEO coarse space starts to
struggle if  $m$ gets too big. This observation aligns with the
theory which suggests that a large divergence can hamper the
effectiveness of the coarse space.

\begin{figure}[t]
	\centering
	\includegraphics[width=0.49\textwidth,trim=1.8cm 0cm 1.8cm 0cm ,clip]{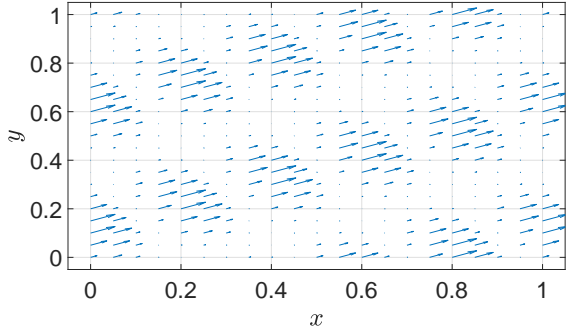}
	\caption{Unidirectional convection field with oscillations: $\mathbf{b}$ given by \eqref{conv:unidirectional-oscillations} with $m=4$.}
    \label{fig:conv_field_unidir}
\end{figure}

\begin{table}[t]
	\centering
	\tabulinesep=1.2mm
	\begin{tabu}{c|ccccc}
		& \multicolumn{5}{c}{$N$} \\
		$m$ & 4 & 16 & 36 & 64 & 100 \\
		\hline
		0 & 53 & 57 & 55 & 54 & 51 \\
		2 & 56 & 85 & 88 & 95 & 78 \\
		4 & 50 & 97 & 143 & 162 & 123 \\
		8 & 56 & 79 & 171 & 133 & 181 \\
	\end{tabu}
	\caption{Iteration counts for the heterogeneous convection--diffusion problem with oscillating unidirectional convection coefficient \eqref{conv:unidirectional-oscillations}, varying the parameter $m$ (which controls the size of $\nabla \cdot \mathbf{b}$) and the number of subdomains $N$, with fixed $h = 1/600$, $b = 1000$ and $a_\text{max} = 50$.}
	\label{Table:Convection_unidirectional_oscillating_heterogeneous_nglob600}
\end{table}

To summarise, the GenEO coarse space is able to handle variations in the convection field to an extent. However, should $\nabla \cdot \mathbf{b}$ become too large then the approach can lose robustness.

\subsection{Time evolution for convection--diffusion problems}
\label{subsec:numerics-timestepping}

Finally, we also include one experiment on time evolution for
convection--diffusion problems. The
underlying time-dependent problem for $u(\mathbf{x},t)$ over the
time-interval $(0,T)$ is given by
\begin{subequations}
\label{eq:convdiff_timedependent}
\begin{align}
\frac{\partial u}{\partial t}-\mathrm{div}( A \nabla u) + \mathbf{b}\cdot\nabla u + \kappa u &= g & & \text{in } \Omega \times (0,T), \\
u &= 0 & & \text{on } \partial\Omega \times (0,T), \\
u &= u_0 & & \text{on } \Omega \times \left\lbrace 0 \right\rbrace.
\end{align}
\end{subequations}
Consider using an implicit scheme in time to yield a semi-discrete
problem. In particular, by way of example, we consider the backward
Euler scheme with time-step $\Delta t > 0$ and set (without loss of generality) $\kappa = 0$, yielding
\begin{subequations}
\label{eq:convdiff_semidiscrete}
\begin{align}
-\mathrm{div}( A \nabla u^{n}) + \mathbf{b}\cdot\nabla u^{n} + \frac{1}{\Delta t} u^{n} &= g^{n} + \frac{1}{\Delta t} u^{n-1} & & \text{in } \Omega, \\
u^{n} &= 0 & & \text{on } \partial\Omega,
\end{align}
\end{subequations}
where $u^{n}(\mathbf{x}) \approx u(\mathbf{x},n \Delta t)$ and $u^0 =
u_0$ with a similar convention for $g$. The canonical expression of the semi-discrete problem at each time-step takes the form
\begin{subequations}
\label{eq:convdiff_timestepping}
\begin{align}
-\mathrm{div}( A \nabla u) + \mathbf{b}\cdot\nabla u + \frac{1}{\Delta t} u &= f & & \text{in } \Omega, \\
u &= 0 & & \text{on } \partial\Omega,
\end{align}
\end{subequations}
for some $f$ and it is this generic problem we consider solving. Other implicit time-stepping schemes will lead to similar PDE problems at each time-step, so this investigation covers the general case.

Since $1/\Delta t > 0$, it could be put completely into the
coefficient  $c^+$, as defined in \eqref{eq:splitc}, and thus used in
the definition of the GenEO coarse space. However, to allow better efficiency in
the case of variable time-stepping it would be better to build a
preconditioner that is independent of  the time-step size. So  we suppose that
a `median time-step' $\Delta t_0$ is chosen and we perform the splitting \eqref{eq:splitc} with
$$ c^+ = 1/\Delta t_0 \quad \text{and} \quad c^{-} =  1/\Delta t  - 1/\Delta t_0. $$
With this choice, the GenEO preconditioner is built---once only,
before time-stepping takes place---incorporating the
$1/\Delta t_{0}$ term, i.e., using the bilinear form
\begin{align*}
a_{\Omega_{i}}(u,v) &= \int_{\Omega_{i}}\left(A\nabla u\cdot{\nabla v}\, + \frac{1}{\Delta t_{0}} u v\right) dx.
\end{align*}
We  study the utility of this choice
by varying parameters $\Delta t_0$ and $\Delta t$.

\begin{table}[t]
	\centering
	\tabulinesep=1.2mm
	\begin{tabu}{c|ccccc}
		 & \multicolumn{5}{c}{$N$} \\
		$\Delta t$ & 4 & 16 & 36 & 64 & 100 \\
		\hline
		1000 & 53/53/44  & 57/56/49  & 55/55/51 &
                54/54/51 & 51/51/49 \\
	        0.1 & 53/53/44 & 57/56/49 & 55/55/51 & 54/54/51 & 51/51/49 \\
		0.001 & 50/50/43  & 53/53/48 & 53/53/50 &
                52/52/50 & 50/50/48\\
	\end{tabu}
	\caption{Iteration counts with $\Delta
          t_{0} = 10$ (left), $0.1$ (middle), and $0.001$ (right) for
          one time-step \eqref{eq:convdiff_timestepping} of backward Euler applied
          to the time-dependent convection--diffusion
          problem with heterogeneous diffusion coefficient \eqref{eq:hetcoeff}
          and unidirectional convection coefficient
          \eqref{conv:zero-div}, varying the time-step
          $\Delta t$ and the number of subdomains $N$, for fixed $h =
          1/600$, $b = 1000$ and $a_\text{max} = 50$.}
	\label{Table:Convection_time-stepping_unidirectional_oscillating_m0n2_heterogeneous_nglob600}
\end{table}
We take $\Omega = (0,1)^2$, $A$ as in \eqref{eq:hetcoeff}, and consider the divergence-free convection coefficient \eqref{conv:zero-div}, fixing $b = 1000$ and $a_\text{max} = 50$.
Results are given in Table~\ref{Table:Convection_time-stepping_unidirectional_oscillating_m0n2_heterogeneous_nglob600}
where we see that the GenEO preconditioner is robust to changes in
both $\Delta t_0$ and $\Delta t$ (with either $\Delta t \le \Delta t_0 $ or $\Delta t > \Delta t_0 $), thus providing
an effective and efficient preconditioner for implicit time-stepping of
non-self adjoint or indefinite time-dependent problems.

\subsection{Alternative two-level preconditioners and GenEO-type coarse spaces}
\label{subsec:alternative-geneo}

The theory in this paper and the numerical results above are for the
classical, additive two-level Schwarz preconditioner with the
GenEO coarse space in \eqref{eq:2-19}. The underlying generalised
eigenproblem in \eqref{eq:2-18} is based on the self-adjoint and positive-definite
subproblem \eqref{eq:2-4}, which inevitably ceases to provide a robust
preconditioner when the indefiniteness or non-self-adjointness is very strong.

Especially for indefinite problems, such as
\eqref{eq:helmholtz}, the \emph{Restricted Additive Schwarz} (RAS)
preconditioner \cite{cai_sarkis}
along with a deflation approach to incorporate
the coarse space has previously been shown to lead  to
significantly lower iteration counts and better robustness with respect
to the indefiniteness \cite{bootland:2021:GCS,bootland:2021:ACS}.

For the remainder of this section we return to problem \eqref{eq:helmholtz},
and start by repeating the
experiments  reported in Table~\ref{Table:Heterogeneous_nglob600} (left),
but this time we use RAS instead of additive Schwarz for the one-level part of the preconditioner,
and we combine this with the coarse solve using a deflation approach. Otherwise the setup is identical to that in Table~\ref{Table:Heterogeneous_nglob600} (left). See, e.g., \cite{bootland:2021:GCS} for more details of this `RAS plus deflation' algorithm.
Results for this `RAS plus deflation' approach are given in Table
\ref{Table:RAS} (right) which can be compared with the corresponding  results from Table~\ref{Table:Heterogeneous_nglob600}  (given again here on the left for convenience). We see that the `RAS plus deflation' algorithm
requires roughly half the number of iterations needed by the standard algorithm,
but it does not yield any significantly improved
robustness with respect to the indefiniteness.
\begin{table}[t]
\centering
\tabulinesep=1.2mm
\begin{tabu}{cc|cccc|cccc}
	& & \multicolumn{8}{c}{$N$} \\
	$\kappa$ & $a_\text{max}$ & 16 & 36 & 64 & 100 & 16 & 36 & 64 & 100 \\
	\hline
	10 & 5 & 18 & 18 & 19 & 19     & 9 & 9 & 9 & 9 \\
	10 & 10 & 18 & 19 & 19 & 19      & 9 & 9 & 9 & 9 \\
	10 & 50 & 18 & 20 & 19 & 19     & 9 & 9 & 9 & 9 \\
	\hline
	100 & 5 &  30 & 26 & 23 & 22     & 11 & 10 & 10 & 9 \\
	100 & 10 & 26 & 24 & 23 & 22     & 11 & 10 & 10 & 9 \\
	100 & 50 & 22 & 22 & 20 & 20     & 9 & 9 & 9 & 9 \\
	\hline
	1000 & 5 & 140 & 177 & 167 & 138     & 62 & 76 & 70 & 57 \\
	1000 & 10 & 118 & 146 & 131 & 108     & 50 & 62 & 55 & 45 \\
	1000 & 50 & 94  & 114 & 106 & 99     & 38 & 46 & 43 & 39 \\
	\hline
	10000 & 5 & 524 & 906 & $\!\!1000+\!\!$ & $\!\!1000+\!\!$ &   297 & 504 & 628 & 753 \\
	10000 & 10 & 457 & 815 & $\!\!1000+\!\!$& $\!\!1000+\!\!$ &    234 & 445 & 548 & 644 \\
	10000 & 50 & 318 & 704 & 888      & $\!\!1000+\!\!$ &    171 & 372 & 464 & 538
\end{tabu}
\caption{Iteration counts for problem \eqref{eq:helmholtz} with
heterogeneous $A$ as in \eqref{eq:hetcoeff}, using the GenEO coarse space based on
\eqref{eq:2-18} and varying $\kappa$, $a_\text{max} > 1$ and $N$
with $h = 1/600$ fixed. Left: Classical two-level, additive Schwarz.
Right: Restricted additive Schwarz and deflation.}
\label{Table:RAS}
\end{table}

In \cite{bootland:2021:GCS} we also compared the GenEO coarse space
based on the eigenproblem \eqref{eq:2-18}, denoted $\Delta$-GenEO
in \cite{bootland:2021:GCS}, with an alternative approach first
introduced for the Helmholtz problem in~\cite{bootland:2021:ACS},
going under the name of $\mathcal{H}$-GenEO.
\begin{definition}\label{def:5-12}
($\mathcal{H}$-GenEO eigenproblem \iggr{for problem \eqref{eq:helmholtz}} ). For each $j=1,\ldots,N$, we define the following generalized eigenvalue problem
\begin{equation}\label{eq:5-12}
\text{Find} \quad (p, \lambda) \in {\tV}_j\backslash \{0\} \times \mathbb{R} \quad \text{such that} \quad      b_{\Omega_{j}}(p,v) = \lambda \,a_{\Omega_{j}}(\Xi_{j}(p), \Xi_{j}(v)),\quad {\rm for}\;\,{\rm all}\;\,v\in \widetilde{V}_{j},
\end{equation}
where $b_{\Omega_j}(\cdot,\cdot)$ \iggr{is the bilinear form corresponding to the indefinite  operator $-\Delta - \kappa$}  and $\Xi_{j}$ is the local partition of unity operator from Definition~\ref{def:2-0}.
\end{definition}

Recalling \eqref{eq:2-18},  we see that the key difference in $\mathcal{H}$-GenEO compared with $\Delta$-GenEO is the
use  of the bilinear form $b_{\Omega_j}(\cdot,\cdot)$ of the underlying
indefinite problem on the left-hand side. Note that this also implies that
the eigenvalues $\lambda$ in \eqref{eq:5-12} {can be negative}.
Thus,
when incorporating eigenvectors into the coarse space using the
threshold $\lambda < \lambda_\text{max}$ we also include all
negative eigenvalues (see again \cite{bootland:2021:GCS} for  details).

In our final set of results, Table~\ref{Table:H-GenEO}, we repeat the
experiments in Table~\ref{Table:RAS}
using the GenEO coarse space based on the $\mathcal{H}$-GenEO
eigenproblem in \eqref{eq:5-12}, again with an otherwise identical
setup. We see that for small $\kappa$ the iteration counts are almost
identical. For large $\kappa$, we see that $\mathcal{H}$-GenEO leads to
much lower iteration counts, such that GMRES
converges in less than $50$ iterations even for $\kappa = 10000$ when
using the `RAS plus deflation' approach. Unfortunately, a theoretical
justification for this excellent performance of the $\mathcal{H}$-GenEO
method is still lacking. We note that for larger $\kappa$ the coarse
space induced by the
$\mathcal{H}$-GenEO eigenproblem in \eqref{eq:5-12} is slightly larger
than that induced by the eigenproblem in \eqref{eq:2-18}
(see \cite[Table 2]{bootland:2021:GCS}). Nonetheless, this
increase is relatively modest compared to the
reduction in iteration counts that can be achieved.

\begin{table}[t]
\centering
\tabulinesep=1.2mm
\begin{tabu}{cc|cccc|cccc}
	& & \multicolumn{8}{c}{$N$} \\
	$\kappa$ & $a_\text{max}$ & 16 & 36 & 64 & 100 & 16 & 36 & 64 & 100 \\
	\hline
	10 & 5 & 18 & 18 & 19 & 19 & 9 & 9 & 9 & 9  \\
	10 & 10 & 18 & 19 & 19 & 19 & 9 & 9 & 9 & 9   \\
	10 & 50 & 18 & 19 & 19 & 19 & 9 & 9 & 9 & 9    \\
	\hline
	100 & 5 & 18 & 19 & 19 & 19 & 9 & 9 & 9 & 9  \\
	100 & 10 & 18 & 19 & 19 & 19 & 9 & 9 & 9 & 9 \\
	100 & 50 & 18 & 19 & 19 & 19 & 8 & 9 & 9 & 9  \\
	\hline
	1000 & 5 & 43 & 36 & 29 & 32 & 12 & 12 & 11 & 12  \\
	1000 & 10 & 36 & 43 & 23 & 24 & 11 & 11 & 10 & 10 \\
	1000 & 50 & 28 & 31 & 24 & 26 & 12 & 11 & 11 & 12  \\
	\hline
	10000 & 5 & 134 & 173 & 192 & 258 & 35 & 35 & 43 & 45   \\
	10000 & 10 & 113 & 147 & 191 & 192 & 21 & 29 & 41 & 34  \\
	10000 & 50 & 92 & 109 & 173 & 182 & 19 & 25 & 48 & 26
\end{tabu}
\caption{Iteration counts for problem \eqref{eq:helmholtz} with
  heterogeneous $A$ as
  in \eqref{eq:hetcoeff}, using the $\mathcal{H}$-GenEO coarse space based on
  \eqref{eq:5-12} and varying $\kappa$, $a_\text{max} > 1$ and $N$
  with $h = 1/600$ fixed. Left: Classical two-level, additive Schwarz.
  Right: Restricted additive Schwarz and deflation.}
\label{Table:H-GenEO}
\end{table}

These results provide an insight into the power of using generalised
eigenproblems to yield coarse spaces for challenging indefinite or
non-self-adjoint problems, which has yet to be fully explored even
numerically. A number of theoretical obstacles remain in order to analyse
the effect of such GenEO-type methods, but theory in this direction may
provide a greater understanding of where such approaches can be
successful and how to define appropriate, problem-dependent
generalised eigenproblems to incorporate the salient properties of the
underlying problem.